\documentclass[12pt]{article}
\usepackage[margin=1.5in]{geometry}  

\usepackage[small, center]{titlesec}
\titlelabel{\thetitle. }

\usepackage{amsmath, amsthm, amssymb} 
\usepackage{amsfonts}			  
\usepackage[dvipsnames]{xcolor}
\usepackage{hyperref}
\usepackage[utf8]{inputenc}
\usepackage[english]{babel}
\usepackage{mathrsfs} 
\usepackage{setspace}
\usepackage{mathtools}
\usepackage{enumitem}
\usepackage[all]{xy}

\newtheorem{thm}{Theorem}[section]
\newtheorem{prop}[thm]{Proposition}

\newtheorem{lem}[thm]{Lemma}

\theoremstyle{definition}
\newtheorem{defn}[thm]{Definition}
\newtheorem{ex}[thm]{Example}

\theoremstyle{remark}
\newtheorem{rmk}[thm]{Remark}

\newcommand{\bC}{\mathbb{C}}

\newcommand{\bN}{\mathbb{N}}

\newcommand{\bR}{\mathbb{R}}

\newcommand{\bZ}{\mathbb{Z}}
\newcommand{\cA}{\mathcal{A}}
\newcommand{\cB}{\mathcal{B}}

\newcommand{\cH}{\mathcal{H}}
\newcommand{\cI}{\mathcal{I}}
\newcommand{\cJ}{\mathcal{J}}
\newcommand{\cK}{\mathcal{K}}

\newcommand{\cS}{\mathcal{S}}

\newcommand{\cU}{\mathcal{U}}

\newcommand{\cZ}{\mathcal{Z}}
\newcommand{\fa}{\mathfrak{a}}

\newcommand{\fg}{\mathfrak{g}}
\newcommand{\fk}{\mathfrak{k}}
\newcommand{\fp}{\mathfrak{p}}

\newcommand{\fh}{\mathfrak{h}}

\newcommand{\fn}{\mathfrak{n}}
\newcommand{\ft}{\mathfrak{t}}

\renewcommand{\Re}{\operatorname{Re}}

\DeclareMathOperator{\tr}{tr}

\DeclareMathOperator{\Aut}{Aut}

\DeclareMathOperator{\Hom}{Hom}

\DeclareMathOperator{\Lie}{Lie}

\DeclareMathOperator{\ad}{ad}
\DeclareMathOperator{\Ad}{Ad}

\DeclareMathOperator{\Id}{Id}
\DeclareMathOperator{\Tr}{Tr}
\DeclareMathOperator{\MC}{\mathsf{MC}}

\newcommand{\nbar}{\overline{n}}

\DeclareMathOperator{\PW}{\mathsf{PW}}

\newcommand{\minK}[1]{A(#1)}
\DeclareMathOperator{\GL}{GL}
\DeclareMathOperator{\SL}{SL}
\DeclareMathOperator{\SO}{SO}

\DeclareMathOperator{\End}{End}

\newcommand{\Khat}{\widehat{K}}

\DeclareMathOperator{\diag}{diag}

\newlength{\oldlineskip}
\numberwithin{equation}{section}

\usepackage{csquotes}

\DeclareMathOperator{\rest}{rest}

\newcommand{\bbar}{\overline{b}}

\newcommand{\tPsi}{\tilde{\Psi}}

\title{Compatible Decomposition of the Casselman Algebra and the Reduced Group $C^*$-algebra of a Real Reductive Group}
\author{Jacob Bradd}
\date{}

\begin{document}
	
\maketitle
\begin{abstract}
For a real reductive group $G$, we investigate the structure of the Casselman algebra $\cS(G)$  and its similarities to the structure of the reduced group $C^*$-algebra $C_r^*(G)$. We demonstrate that the two algebras are assembled from very similar elementary components in a compatible way. In particular, we prove that the two algebras have the same $K$-theory when restricted to a finite set of $K$-types, which is a refinement of the Connes-Kasparov isomorphism.
\end{abstract}

\section{Introduction}
Let $G$ be a real reductive Lie group and let $K$ be a maximal compact subgroup. The elements of the reduced group $C^*$-algebra $C_r^*(G)$ consist of (generalized) functions on $G$ that are roughly in $L^2(G)$ (cf. \cite{cowlingKS}). The Casselman algebra $\cS(G)$ (see Definition \ref{defn: cassalg}) consists of very rapidly decreasing functions on $G$, and is a much smaller (Fr{\'e}chet) algebra. Despite the differences, we shall prove that these two algebras are assembled in a very similar way from very similar elementary components. In particular, these elementary components have identical $K$-theory, and as a result, $C_r^*(G)$ and $\cS(G)$ have the same $K$-theory, too, when the $K$-types are restricted to a finite set.

Given $F \subset \widehat{K}$, there is a ($K$-finite) function $p_F \in C(K)$ which acts on any $K$-representation by projection onto the $K$-types in $F$. Now $\cS(G)$ and $C_r^*(G)$ are $K \times K$-representations, and we write
\[\cS(G,F) = p_F \cS(G) p_F, \quad C_r^*(G,F) = p_F  C_r^*(G) p_F.\]
That is, $\cS(G,F)$ (resp $C_r^*(G,F)$) is the projection of $\cS(G)$ (resp. $C_r^*(G)$) onto the $K \times K$-types in $F \times F$.

The representations of $K$ may be equipped with a notion of length, due to Vogan \cite{voganmainpaper}. The main theorem of this paper is that, when $R > 0$ and when $F$ is the set of $K$-types with length at most $R$, the inclusion $\cS(G,F) \hookrightarrow C_r^*(G,F)$ induces an isomorphism
\begin{equation}\label{eq: CKintro}
	K_*(\cS(G,F)) \xrightarrow{\cong} K_*(C_r^*(G,F)).
\end{equation}
Here, we are using the $K$-theory of Fr{\'e}chet algebras defined by Phillips \cite{phillipsktheory} (the $K$-functor is written as $RK$ there).

The isomorphisms in \eqref{eq: CKintro} refine the Connes-Kasparov isomorphism  (see \cite[(4.20)]{baumconneshigson},  \cite[Section 2.4]{valetteBCsurvey}), which is equivalent to the assertion that the inclusion of $\cS(G)$ into $C_r^*(G)$ induces an isomorphism in $K$-theory. Our result can be used to check the original Connes-Kasparov isomorphism.

It is natural to try to explain the isomorphism \eqref{eq: CKintro} as a manifestation of an Oka principle. In the theory of several complex variables, Grauert \cite{grauert1, grauert2,grauert3} proved that topological vector bundles on Stein spaces can be given a holomorphic structure, unique up to homotopy. This has an interpretation in $K$-theory (due to Novodvorskii \cite{novodvorskii}), namely that the $K$-theory of a commutative Banach algebra is isomorphic to the topological $K$-theory of its Gelfand spectrum. See \cite{2021braddhigson} for an exposition and references.

The isomorphism \eqref{eq: CKintro} also has such an interpretation. Indeed, the representation theory of $C_r^*(G)$ is related to the unitary representations of $G$ (specifically, the \textit{tempered} representations). On the other side, the representation theory of $\cS(G)$ is related to all (admissible) representations (see \cite{2014bernsteinkrotz}). Moreover, there is a Fourier transform on $\cS(G)$ which uses the \textit{nonunitary} principal series, depending on complex parameters, and the Fourier transform of an element of $\cS(G)$ depends holomorphically on these parameters. We can then think of the map $\cS(G) \to C_r^*(G)$ as a restriction map from holomorphic functions on the nonunitary principal series to continuous functions on the tempered dual.

The proof of the isomorphism \eqref{eq: CKintro} is largely based on techniques due to Delorme \cite{delormePW} that are used in his characterization of the Fourier image of $C_c^\infty(G)$ (this description is the ``Paley-Wiener theorem'', first established for general real reductive groups by Arthur \cite{arthur}). These techniques can be adapted to $\cS(G)$ with little change.

Here are the main steps in the argument.

Let $P$ be a cuspidal parabolic subgroup of $G$, let $P = MAN$ be its Langlands decomposition, and let $\sigma$ be a square-integrable representation of $M$. Associated to the pair $(P,\sigma)$ is a Hilbert space $\cH_\sigma$ and a series of $G$-representations $(\pi^P_{\sigma,\lambda}, \cH_\sigma)$ for each $\lambda \in \fa^*$ (where $\fa$ is the complexification of the Lie algebra $\fa_0$ of $A$), called the (nonunitary) principal series. We denote the corresponding $(\fg, K)$-modules by $(\pi^P_{\sigma,\lambda}, I_\sigma)$. 

Given $\phi \in \cS(G,F)$ and $v \in p_F I_\sigma$, we define $\pi_{\sigma,\lambda}^P(\phi)v \in p_FI_\sigma$ by
\[\pi^P_{\sigma,\lambda}(\phi)v = \int_G \phi(g) \pi^P_{\sigma,\lambda}(g) v \, dg.\] 
The map $\lambda \mapsto \pi^P_{\sigma,\lambda}(\phi)$ is a holomorphic function from $\fa^*$ to the finite-dimensional space $\End(p_F I_\sigma)$. Moreover, if we set (for a finite-dimensional normed vector space $V$)
\begin{align*}
	\PW(\fa^*, V)= \{f: \fa^* \to V:& \quad\text{$f$ is holomorphic and } \\
	&\sup_{|\Re \lambda| \leq k} (1+|\lambda|)^N \|f(\lambda)\| < \infty \, \text{ for all } N,k \in \bN\},
\end{align*} 
then we obtain from $\pi^P_{\sigma,\lambda}$ a continuous map
\[\pi_\sigma: \cS(G,F) \to \PW(\fa^*, \End(p_F I_\sigma)).\]

Let $\minK{\sigma}$ denote the set of minimal $K$-types of $I_\sigma$ (that is, the $K$-types of $I_\sigma$ of minimal length). A deep theorem of Vogan \cite{voganmainpaper} states that the set $\minK{\sigma}$ determines the pair $(P, \sigma)$ up to $G$-conjugacy, and that the sets $\minK{\sigma}$ partition $\Khat$. Accordingly, we can totally order the $G$-conjugacy classes $[P,\sigma]$ using the sets $\minK{\sigma}$ and the common lengths of their elements. Choose representatives $(P_n, \sigma_n)$ so that
\[[P_1, \sigma_1] < [P_2, \sigma_2] < \cdots \]
We then define ideals
\[0 = J_0 \subset J_1 \subset \cdots \subset J_N = \cS(G,F)\]
by the property that $\pi_{\sigma_m}(J_n) = 0$ for $m > n$. Thus, $J_1$ consists of functions $\phi \in \cS(G,F)$ which vanish on every principal series other than the spherical principal series $I_{\sigma_1}$ (whose minimal $K$-type is the trivial $K$-type), while  $J_2$ consists of functions that vanish on all principal series other than $I_{\sigma_1}$ and $I_{\sigma_2}$, and so on. By definition, $\pi_{\sigma_n}$ is injective on the subquotient $\cJ_n/\cJ_{n-1}$.   

We define ``Morita equivalence'' for Fr{\'e}chet algebras $\cA$ in the narrow sense that if $p$ is a projection in (the ``multiplier algebra'' of) $\cA$ such that $\overline{\cA p \cA} = \cA$, then $\cA \sim p\cA p$. We will prove in Section \ref{sec: ktheory} that if $\cA\sim p\cA p$, then the inclusion $p\cA p \hookrightarrow \cA$ induces an isomorphism in $K$-theory. Making use of Delorme's techniques (adapted to $\cS(G)$), we obtain a ``Morita equivalence''
\[\cJ_n/\cJ_{n-1}\sim \PW(\fa^*, \End(p_{\minK{\sigma_n}} I_{\sigma_n}))^{W_{\sigma_n}},\]
where $W_{\sigma_n}$ is a certain finite group acting on $\PW(\fa^*, \End(p_{\minK{\sigma_n}} I_{\sigma_n}))$ in a fairly simple way (in particular, the action is mostly induced by an action on $\fa^*$).

We define $J_n \subset C_r^*(G,F)$ similarly, and we have (using results of \cite{clarecrisphigson})
\[J_n/J_{n-1} \sim C_0(i\fa^*_0, \End(p_{\minK{\sigma_n}} I_{\sigma_n}))^{W_{\sigma_n}}.\]
Moreover, the inclusion
\[\PW(\fa^*, \End(p_{A(\sigma)}I_\sigma)))^{W_\sigma} \hookrightarrow  C_0(i\fa^*_0, \End(p_{\minK{\sigma}} I_{\sigma}))^{W_\sigma}\]
induces an isomorphism in $K$-theory by a simple homotopy argument.

It follows that the inclusion
\[\cJ_n /\cJ_{n-1} \hookrightarrow J_{n}/J_{n-1}\]
induces an isomorphism in $K$-theory. The isomorphism \eqref{eq: CKintro} is established by a series of $6$-term exact sequence and five-lemma arguments.

To summarize, we apply techniques of Delorme \cite{delormePW} and the results of Clare-Crisp-Higson \cite{clarecrisphigson} to decompose $\cS(G,F)$ and $C_r^*(G,F)$ into elementary components, which are Morita equivalent to fairly simple function spaces. These have isomorphic $K$-theory by a simple homotopy argument, which can be regarded as a simple application of the Oka principle.

Our filtrations are similar to those appearing in the recent work of Afgoustidis \cite{afgoustidisCK}, who has provided a proof of the Connes-Kasparov isomorphism using the Cartan motion group and the Mackey analogy, generalizing Higson \cite{higsonmackeyanalogy} for complex groups. He defines ideals in $C_r^*(G)$, corresponding to the sets $A(\sigma)$, that are the same as ours.
However, we compare $C_r^*(G)$ not to the $C^*$-algebra of the motion group, but to $\cS(G)$, and on $\cS(G)$ our ideals are inspired by the ideals defined by Delorme in \cite[Proposition 2]{delormePW} (in fact, it is possible to use Delorme's ideals directly, employing an ``induction on $K$-type length'' argument, but the refinement using minimal $K$-types is more convenient for this purpose).

The structure of the paper is as follows. We first state the main theorem in Section 2, and then develop some basic Fr{\'e}chet algebra $K$-theory in Section 3. We provide some background on representation theory in Section 4, and in Section 5 we define the ideals, as above, and state four theorems that amount to the ``Morita equivalence'' outlined above. In Section 6, we reduce the main theorem to these four theorems. In Section 7, we reduce those four theorems to a ``Factoring Theorem'' that is analogous to \cite[Proposition 1]{delormePW}. Finally, in Section 8, we prove this Factoring Theorem (and a ``Divisibility Theorem'') by adapting Delorme's proof of \cite[Propositions 1,2]{delormePW} to $\cS(G)$.

\section{Preliminaries}

We restrict our class of real reductive groups to those considered by Knapp \cite{knappcommutativity} (see also \cite{delormeclozel}). These are closed subgroups $G \subset \GL(n,\bR)$ with finitely many connected components such that $\fg_0$ is reductive and, if $G_\bC$ denotes the analytic subgroup of $\GL(n,\bC)$ corresponding to $\fg$, and if $Z_\bC(G)$ denotes the centralizer of $G$ in $\GL(n,\bC)$, then
\[G \subset G_\bC \cdot Z_\bC(G).\]
These groups have the advantage of satisfying hypotheses of Harish-Chandra, Knapp-Stein and Vogan (see Clozel and Delorme \cite[Section 1.2]{delormeclozel} for more precise statements). Any group of real points of a connected reductive algebraic group defined over $\bR$ is in this class of groups.

Let $K$ denote a choice of maximal compact subgroup. We write $\Khat$ for the (isomorphism classes of) irreducible unitary representations of $K$. We will write $\gamma \in \Khat$ to mean a (fixed) representative $(\gamma, V_\gamma)$ of an element of $\Khat$.

We denote, for example, $\fg$ to be the \textit{complexified} Lie algebra of $G$, and $\fg_0 = \Lie(G)$ to denote the corresponding real Lie algebra.  We fix a Cartan decomposition $\fg_0 = \fk_0 + \fp_0$ with corresponding Cartan involution $\theta$, and choose a maximal abelian Lie subalgebra $\fa_{0,\min}$ of $\fp_0$ with corresponding $A_{\min} = \exp(\fa_{0,\min})$.

We fix an invariant bilinear form $B$ on $\fg_0$ (for example, the Killing form $B(X,Y) = \tr(\ad_X \ad_Y)$ in the semisimple case), which is negative definite on $\fk_0$ and positive definite on $\fp_0$, hence $\fa_0$. We define the inner product $\langle \cdot, \cdot \rangle$ on $\fg$ by
\begin{equation}\label{eq: innerproduct}
	\langle X,Y\rangle = -B(X, \theta Y).
\end{equation}

We also have the decomposition
\[G = KA_{\min} K.\]
Writing $g = k_1 e^X k_2$ for $X \in \fa_{0,\min}$, we set $\|g\| = e^{\|X\|}$. This definition depends only on $g$, and defines a norm on $G$ in the sense of \cite[2.A.2]{wallachbook1}.

\begin{defn}\label{defn: cassalg}
	The Casselman algebra $\cS(G)$ is the space
	\begin{align*}
		\cS(G) = \{\phi \in C^\infty(G) \mid \; \|g\|^N (L_u R_v \phi)(g) \in L^1(G) \; \forall u,v \in \cU(\fg), N \in \bN\}.
	\end{align*}
	Here, $L_u$ (resp. $R_v$) denotes the left-regular (resp. right-regular) action of the enveloping algebra $\cU(\fg)$ on $C^\infty(G)$.
	
	The Casselman algebra is a Fr{\'e}chet algebra with seminorms defined as follows. Fix an ordered basis $X_1, \ldots, X_{\dim G}$ of $G$. We set
	\begin{equation}\label{eq: SG seminorms}
		\|\phi\|_{\cS(G),N,k} = \sum_{|I|,|J|\leq k}\int_G (1+\|g\|)^N |L_{X^I}R_{X^J} \phi|\,dg,
	\end{equation}
	where $I$ and $J$ are multi-indices.
\end{defn}
\begin{defn}
	The reduced group $C^*$-algebra $C_r^*(G)$ is the completion of $L^1(G)$ with
	respect to the norm
	\[\|f\|_{C_r^*(G)} = \sup_{\|h\|_{L^2(G)}=1} \|f * h\|.\]
	That is, $C_r^*(G)$ is the closure of $L^1(G)$ embedded into $\cB(L^2(G))$ under the
	left-regular representation.
\end{defn}

Note that $\cS(G)$ is a subset of $L^1(G)$ and therefore a subset of $C_r^*(G)$. 

\begin{defn}\label{defn: projectionontoF}
	Given $\gamma \in \Khat$, define $p_\gamma \in C^\infty(K)$ by
	\[p_\gamma(k) = \overline{\Tr(\gamma(k^{-1}))}.\]
	Given an $K$-module $(\pi, E)$, then $\pi(p_\gamma)$ is precisely the projection onto the $\gamma$-isotypical component of $E$. Given a finite subset $F \subset \Khat$, we write 
	\[p_F = \sum_{\gamma \in F} p_\gamma.\]
	Treating $\cS(G)$ and $C_r^*(G)$ as $K \times K$-modules, we define
	\[\cS(G,F) = p_F \cS(G) p_F, \quad C_r^*(G,F) = p_F C_r^*(G) p_F.\]
\end{defn}
We will make use of the projections $p_F$ extensively. In particular, given a $K$-representation $E$, we will write $p_F E$ to denote the projection of $E$ onto the $K$-types in $F$ (instead of more common notation such as $E(F)$, which appears for example in \cite{wallachbook1}).

\begin{defn}
	We choose a Cartan subalgebra $\ft_K \subset \fk$, and fix a positive root system $\Delta^+(\fk, \ft_{K})$. We write $\rho_c$ for the half-sum of these positive roots. Given a $K$-type $\gamma \in \Khat$ with a highest weight $\overline{\gamma}$, we define the ``length'' of $\gamma$ by
	\[\|\gamma\| = \langle \overline{\gamma} + 2\rho_c, \overline{\gamma} + 2\rho_c\rangle,\]
	which is independent of the highest weight chosen (in the disconnected case). This is as in \cite[Definition 5.4.18]{voganbook} (see also \cite[Section X.2]{knappvogan}).
\end{defn}

The aim of this paper is to prove the following.
\begin{thm}\label{thm: mainthm}
	For each $R \geq 0$, setting $F = \{\gamma \in \Khat: \|\gamma\| \leq R\}$, the inclusion map $\cS(G,F) \to C_r^*(G,F)$ induces an isomorphism in $K$-theory.
\end{thm}
Here we note that $K$-theory for $\cS(G,F)$ is defined in the sense of Phillips \cite{phillipsktheory}. This notion of $K$-theory is not equivalent to the usual notion (i.e. using stabilization via $\varinjlim M_n(\bC)$ or $\cK(\cH)$) because $\cS(G,F)$ is not a ``good'' Fr{\'e}chet algebra, in the sense that the subset of invertible elements in the unitization $\cS(G,F)^+$ is not an open subset (see \cite[A.1.2]{1990bost} for the notion of good algebra).

Finally, we note that the statement of Theorem \ref{thm: mainthm} does not directly imply that the map
\[K_*(\cS(G)) \to K_*(C_r^*(G))\]
is an isomorphism. However, the isomorphism
\[\varinjlim_F K_*(\cS(G,F)) \xrightarrow{\cong} \varinjlim_F K_*(C_r^*(G)) \cong K_*(C_r^*(G))\]
supplied by Theorem \ref{thm: mainthm} does factor through the above map. Therefore, the map $K_*(\cS(G)) \to K_*(C_r^*(G))$ is surjective, which is considered the ``main'' half of the Connes-Kasparov isomorphism (split-injectivity of the Dirac induction map is due to Kasparov \cite{kasparov88}).

\section{Fr{\'e}chet algebra $K$-theory and Morita equivalence}\label{sec: ktheory}

\subsection{Mapping cones}
We recall the notion of mapping cones for Fr{\'e}chet algebras and the corresponding $6$-term exact sequence in $K$-theory. Here we use the $K$-theory and results of Phillips \cite{phillipsktheory}, and we will write $K_i$ instead of $RK_i$.
\begin{defn}
	The mapping cone of a continuous Fr{\'e}chet algebra homomorphism $f: A \to B$ is the Fr{\'e}chet algebra
	\[\MC(f) = \{(\gamma, a) \in C([0,1],B) \oplus A:  \gamma(0) = f(a),\, \gamma(1) = 0 \}.\]
\end{defn}
\begin{lem}\label{lem: mappingconething}
	The map $f: A \to B$ induces an isomorphism in $K$-theory if and only if $\MC(f)$ has zero $K$-theory.
\end{lem}
\begin{proof}
	We have a short exact sequence
	\[0 \to S(B) \to \MC(f) \to A \to 0,\]
	where $S(B)$ is the suspension of $B$,
	\[S(B) = \{\phi: [0,1] \to B: \phi(0) = \phi(1) = 0\}.\]
	By Theorems 6.1 and 5.5 of \cite{phillipsktheory}, we obtain to a $6$-term exact sequence
	\[\xymatrix{
		K_0(\MC(f)) \ar[r]& K_0(A) \ar[r] &K_0(B) \ar[d] \\
		K_1(B) \ar[u] & \ar[l] K_1(A) & \ar[l] K_1(\MC(f)).
	}\]
	The lemma follows immediately from this exact sequence.
\end{proof}

\subsection{Morita equivalence}
We recall a theorem regarding Morita equivalence for Banach algebras, due to Lafforgue and recorded by Paravicini \cite{paravicini}.

\begin{defn}
	A Banach algebra (or Fr{\'e}chet algebra) $A$  is said to be \textit{non-degenerate} if the multiplication map $A \times A \to A$ has dense range.
\end{defn}
We define the \textit{multiplier algebra} $M(A)$ of a Banach algebra $A$ to be the algebra of double centralizers of $A$. That is, $M(A)$ consists of pairs $(L,R)$ of homomorphisms $A \to A$ which satisfy $aL(b) = R(a)b$ for $a,b \in A$. Elements of $M(A)$ act on the left of $A$ via $L$ and on the right via $R$. That is, if $T = (L,R) \in M(A)$, then $Ta = L(a)$ and $aT = R(a)$.
\begin{defn}
	Let $A$ be a Banach algebra, and let $p$ be an idempotent in the multiplier algebra $M(A)$.	Then $p$ is said to be a \textit{full idempoten}t if $ApA$ is dense in $A$.
\end{defn}

\begin{thm}[cf. {\cite[Proposition 4.5 and Theorem 4.25]{paravicini}}]\label{thm: banachmorita}
	Let $A$ be a non-degenerate Banach algebra and let $p \in M(A)$ be a full idempotent. The inclusion map $pAp \hookrightarrow A$ induces an isomorphism in $K$-theory.
\end{thm}
\begin{rmk}
	In the generality of \cite[Proposition 4.5 and Theorem 4.25]{paravicini}, it is not explicitly stated that the isomorphism between $K_*(pAp)$ and $K_*(A)$ is induced by inclusion. However, we can apply Paravicini's theorem to the mapping cone $\MC$ of the inclusion map to see that $K_*(\MC) \cong K_*(p\MC p)$. As $p\MC p$ is the cone of $pAp$, which is a contractible algebra, it follows that $\MC$ vanishes in $K$-theory, which implies Theorem \ref{thm: banachmorita}.
\end{rmk}
Now let $A$ be a non-degenerate Fr{\'e}chet algebra. We follow the convention in \cite{phillipsktheory}. That is, we assume that $A$ is an inverse limit of Banach algebras $A_n$, such that the associated homomorphisms $\pi_{m,n}: A_m \to A_n$ and $\pi_n: A \to A_n$ have dense range. Note that $A_n$ is non-degenerate for each $n$, because $A_n A_n$ contains $\pi_n(AA)$, which is dense in $\pi_n(A)$ and hence $A_n$.  We will write $A = \varprojlim A_n$ to present a Fr{\'e}chet algebra $A$ as an inverse limit of Banach algebras $A_n$ under this convention.

\begin{defn}
	Let $A = \varprojlim A_n$ be a Fr{\'e}chet algebra. A full idempotent $p$ of $A$ will refer to a pair of idempotent continuous linear maps $p_L, p_R: A \to A$ such that
	\begin{enumerate}
		\item $ap_L(b) = p_R(a)b$ for all $a, b \in A$,
		\item For each $n$, there exists $p_n \in M(A_n)$ such that 
		\begin{enumerate}[label=(\roman*)]
			\item $p_n\pi_n(a) = \pi_n(p_L(a))$
			\text{ and }$\pi_n(a)p_n = \pi_n(p_R(a))$.
			\item $\pi_{m,n}(p_m a) = p_n \pi_{m,n}(a)$\text{ and }$\pi_{m,n}(a p_m) = \pi_{m,n}(a)p_n$.
		\end{enumerate}
		
		\item $Ap_L(A)$ is dense in $A$.
	\end{enumerate} 
	As usual, $p$ acts on the left of $A$ by $p_L$, and on the right by $p_R$.
\end{defn}

\begin{thm}\label{thm: frechetmorita}
	Let $A = \varprojlim A_n$ be a Fr{\'e}chet algebra, where $\{A_n\}$ is an inverse system of Banach algebras such that the homomorphisms $\pi_n: A \to A_n$ have dense range. If $p$ is a full idempotent of $A$, then the inclusion map $p A p \to A$ induces an isomorphism in $K$-theory.
\end{thm}
\begin{proof}
	Note that $A_n p_n A_n$ is dense in $A_n$, because the former contains $\pi_n(A p A)$, which is dense in $\pi_n(A)$ and hence dense in $A_n$. By Theorem \ref{thm: banachmorita}, this means that $K_*(p_n A_n p_n) \cong K_*(A_n)$.
	
	Let $\MC$ be the mapping cone of $pAp \to A$, and let $\MC_n$ be the mapping cone of $p_nA_np_n \to A_n$. Then $\pi_{m,n}$ induces a homomorphism $\pi_{m,n}: \MC_{m} \to \MC_n$ and similarly $\pi_n$ induces a homomorphism $\pi_n: \MC \to \MC_n$. Moreover, $\MC = \varprojlim \MC_n$. According to \cite[Theorem 6.5]{phillipsktheory}, we have the short exact sequence
	\[0 \to \varprojlim{}^1 K_{1-*}(\MC_n) \to K_*(\MC) \to \varprojlim K_*(\MC_n) \to 0.\]
	As $K_*(\MC_n) = 0$ by Theorem \ref{thm: banachmorita}, the above sequence implies $K_*(\MC) = 0$.
\end{proof}

\section{Representation theory background}
\subsection{Notation}

We choose a \textit{standard} positive system of restricted roots for $\Delta(\fg_0, \fa_{0,\min})$, denoted $\Delta^+$.  Let $\fn_{0,\min} = \bigoplus_{\alpha \in \Delta^+} \fg_{0,\alpha}$, where $\fg_{0,\alpha}$ denotes the corresponding restricted root space to the root $\alpha \in \Delta(\fg_0, \fa_{0,\min})$, and write $N_{\min} = \exp(\fn_{0,\min})$. The corresponding Iwasawa decomposition is
\[G = KA_{\min}N_{\min}.\]
We write $k: G \to K$ and $a: G \to A_{\min}$ for the corresponding projections. 

The standard minimal parabolic subgroup of $G$ is denoted 
\[P_{\min{}} = M_{\min}A_{\min}N_{\min},\]
where $M_{\min} = Z_K(\fa_{\min})$ denotes the centralizer of $\fa_{\min}$ in $K$. We will only consider parabolic subgroups which contain $A_{\min}$, and these are denoted as $P = MAN$, where $N$ is the unipotent radical of $P$, and $MA = P \cap \theta(P)$ is the Levi subgroup of $P$. The standard parabolic subgroups are those containing $P_{\min}$. We write $\Delta^+_P$ to denote the roots with respect to $\fa_0$ appearing in the decomposition 
\[\fn_0 = \bigoplus_{\alpha \in \Delta^+_P} \fg_{0,\alpha}.\] 
We write 
\[\fa^*_{0,P,+} = \{\lambda \in \fa_0^*: \langle \lambda, \alpha\rangle > 0 \text{ for all }\alpha \in \Delta^+_P\} \subset \fa_0^*\]
for the corresponding (open) Weyl chamber. We write $\fa^*_{P,+}$ for elements $\lambda \in \fa^*$ such that $\Re \lambda \in \fa^*_{0, P, +}$.

We write $\log: A \to \fa_0$ for the inverse of the exponential map. Given $\lambda \in \fa^*$, we set
\[a^\lambda = e^{\lambda(\log a)}.\]

\begin{defn}\label{defn: cuspidal parabolics}
	A parabolic subgroup $P = MAN$ is \textit{cuspidal} if there exists a Cartan subgroup $T \subset M$ contained entirely within $K \cap M$. Let $\ft_0 = \Lie(T)$.
\end{defn}

\begin{defn}\label{defn: HC parameter of sigma}
	We will write $\widehat{M}_d$ to denote the isomorphism classes of square-integrable representations of $M$ (\cite[1.3.2]{wallachbook1}). When we write $\sigma \in \widehat{M}_d$, we refer to a fixed representative of the corresponding isomorphism class. We refer to such elements as \textit{discrete series representations} of $M$.
	
	Given $\sigma \in \widehat{M}_d$, we use $\Lambda_\sigma \in i \ft_0^*$ to denote the Harish-Chandra parameter of $\sigma\vert_{M_0}$, where $M_0$ denotes the connected component of $M$ at the identity (see \cite[Theorem 9.20]{knappoverview}).
\end{defn}

\subsection{The Paley-Wiener space and some representation theory}

In order to prove Theorem \ref{thm: mainthm}, we use a notion of Fourier transform on real reductive groups which apply to elements of $\cS(G,F)$.

\begin{defn}
	A cuspidal pair is a pair $(P, \sigma)$ consisting of a cuspidal parabolic subgroup $P = MAN$ and a discrete series representation $\sigma \in \widehat{M}_d$ of $M$.
\end{defn}

\begin{defn}Let $(P, \sigma)$ be a cuspidal pair, let $V_\sigma$ be a Hilbert space representative for $\sigma$, and let $V_\sigma^\infty$ denote the corresponding smooth vectors. We define the Hilbert space $\cH_\sigma$ as the completion of
	\[\{\varphi: K \xrightarrow{C^\infty} V^\infty_\sigma \mid \varphi(mk) = \sigma(m)^{-1} \varphi(k) \text{ for all } m \in M \cap K, k \in K\}\]
	with respect to the inner product $\langle \varphi,\psi \rangle = \int_K \langle \varphi(k), \psi(k)\rangle_{V_\sigma}\, dk$.
	
	We write $I_\sigma$ for the space of $K$-finite elements of $\cH_\sigma$. Given $\lambda \in \bC$, we define for $g \in G$,
	\begin{equation}\label{eq: princp series defn}
		\pi^P_{\sigma,\lambda}(g) \in \cB(\cH_\sigma), \quad (\pi^P_{\sigma,\lambda}(g)\varphi)(k) = a_P(g^{-1}k)^{-(\lambda+\rho_P)} \varphi(k_P(g^{-1}k)),
	\end{equation}
	where $a_P: P \to A$ is the projection onto $A$, and $k_P: G \to K$ is a choice of element in the decomposition $G = KP$ (unique up to an element of $K \cap M$), and $\rho_P = \frac{1}{2} \sum_{\alpha \in \Delta_P^+} (\dim \fg_{0,\alpha})\alpha$. The $G$-representations $(\pi^P_{\sigma,\lambda}, \cH_\sigma)$ and corresponding $(\fg, K)$-modules $(\pi^P_{\sigma,\lambda}, I_\sigma)$ are known as the \textit{principal series representations} corresponding to the pair $(P,\sigma)$ (see \cite[5.2]{wallachbook1}).
\end{defn}

\begin{defn}
	Fix a finite set $F \subset \widehat{K}$. Given a cuspidal pair $(P, \sigma)$, define
	\[\pi_\sigma = \pi_\sigma^P: \cS(G,F) \to C(\fa^*, \End(p_F I_\sigma))\]
	by
	\[\pi_\sigma(\phi)(\lambda)v = \pi^P_{\sigma,\lambda}(\phi)v = \int_G \phi(g) \pi^P_{\sigma, \lambda}(g)v\,dg,\]
	for each $\lambda \in \fa^*$ and $v \in p_F I_\sigma$. The above integral converges from the proof of Lemma \ref{lem: fourier transform S(G)} below.
\end{defn}
\begin{defn}
	Given a Euclidean vector space $V_0$ with complexification $V$, the \textit{Paley-Wiener space} of $V$ is defined to be
	\begin{align*}
		\PW(V)= \{f: V \to \bC:& \quad\text{$f$ is holomorphic and } \\
		&\sup_{\|\Re \lambda\| \leq k} (1+|\lambda|)^N |f(\lambda)| < \infty \, \text{ for all } N,k \in \bN\}.
	\end{align*} 
	The space $\PW(V)$ is a Fr{\'e}chet algebra with respect to the norms
	\[\|f\|_{\PW(V), N, k} = \sup_{|\Re \lambda|\leq k}(1+|\lambda|)^N |f(\lambda)|.\]
\end{defn}
Note that we will often regard $\fa^*$ as the complexification of $\fa^*_0$, which is Euclidean by use of the inner product \eqref{eq: innerproduct}.
\begin{lem}\label{lem: fourier transform S(G)}
	Fix a cuspidal pair $(P, \sigma)$. For each $\phi \in \cS(G,F)$ and vectors $v, w \in p_F I_{\sigma}$, the map $\lambda \mapsto \langle \pi^P_{\sigma, \lambda}(\phi)v, w\rangle$ defines an element of $\PW(\fa^*)$. That is, 
	\[\pi_\sigma(\cS(G,F)) \subset \PW(\fa^*, \End(p_F I_\sigma)).\]
	Moreover, the map $\pi_\sigma: \cS(G,F) \to \PW(\fa^*, \End(p_F I_\sigma))$ is a continuous homomorphism between Fr{\'e}chet algebras.
\end{lem}
\begin{proof}
	The proof for $C_c^\infty(G)$ in place of $\cS(G)$ is given in \cite[Lemma 1]{delormePW}. We shall adapt the argument given there. We use the following estimate from \cite[(1.25)]{delormePW} (which we have relaxed slightly):
	\[\|\pi_{\sigma, \lambda}^P(g)\| \leq \|g\|^{2 |\Re \lambda|}.\]
	Then we see that, for $\varphi, \psi \in I^P_{\sigma, \lambda}$,
	\[|\langle \pi_{\sigma, \lambda}^P(\phi) v, w\rangle| \leq \int_G|\phi(g)| \|g\|^{2|\Re \lambda|} \|v \|\|w\| \, dg.\]
	Therefore,
	\[\sup_{|\Re \lambda| \leq k}|\langle \pi_{\sigma, \lambda}^P(\phi) \varphi, \psi\rangle| \leq \int_G |\phi(g)| \|g\|^{2k} \|v\|\|w\| \, dg, \]
	which is finite by the definition of $\cS(G)$ (in particular, $\pi^P_{\sigma, \lambda}(\phi)$ is well-defined). From the definition \eqref{eq: princp series defn} of $\pi^P_{\sigma, \lambda}(g)$, we see that $\langle \pi_{\sigma, \lambda}(\phi) v, w \rangle$ is a holomorphic function in $\lambda$.
	
	Now fix $N \in \bN$. Set $\fh = \ft + \fa$ (recall Definition \ref{defn: cuspidal parabolics}), which is a Cartan subalgebra of $\fg$. Let $W(\fg, \fh)$ denote the corresponding Weyl group. From \cite[(1.27)]{delormePW} there exist $Q_1, \ldots, Q_r \in \bC[\fh^*]^{W(\fg, \fh)}$ such that
	\begin{equation}\label{eq: fourier transform S(G) HC est}
		(1+ |\nu|^2)^N \leq  |Q_1(\nu)| + \cdots + |Q_r(\nu)|
	\end{equation}
	for $\nu \in \fh^*$. Let $\cZ(\fg)$ denote the center of the enveloping algebra $\cU(\fg)$ of $\fg$. Choosing $z_1, \ldots, z_r \in \cZ(\fg)$ corresponding to $Q_i$ via the Harish-Chandra isomorphism (see \cite[Theorem 3.2.3]{wallachbook1}), then
	\[\pi_{\sigma, \lambda}^P(L_{z_i}\phi) = \pi_{\sigma,\lambda}^P(z_i) \pi_{\sigma,\lambda}^P(\phi) = Q_i(\Lambda_\sigma + \lambda)\pi_{\sigma,\lambda}^P(\phi),\]
	where $\Lambda \in i\ft^*_0$ is as in Definition \ref{defn: HC parameter of sigma}.
	
	Applying \eqref{eq: fourier transform S(G) HC est} to $\nu = \Lambda_\sigma + \lambda$, we have
	\begin{align*}
		\sup_{|\Re \lambda | \leq k} (1 + |\Lambda_\sigma|^2 + |\lambda|^2)^N &|\langle \pi_{\sigma, \lambda}^P(\phi) v, w\rangle|
		\\ \leq & \sum_{i=1}^r \sup_{|\Re \lambda | \leq k} |\langle \pi_{\sigma, \lambda}^P( L_{z_i}\phi) v, w\rangle| 
		< \infty
	\end{align*}
	As $\sigma$ is fixed, the above is equivalent to the condition defining $\PW(\fa^*)$. From the definition of the topologies defined for $\cS(G)$ and $\PW(\fa^*, \End(p_F I_\sigma))$, the above estimate proves that $\pi_\sigma$ is continuous. The fact that $\pi_\sigma$ is an algebra homomorphism follows from the identity
	\[\pi_{\sigma,\lambda}^P(\phi_1 * \phi_2) = \pi_{\sigma,\lambda}^P(\phi_1) \pi_{\sigma,\lambda}^P(\phi_2)\]
	for $\phi_1, \phi_2 \in \cS(G)$, which holds for any $G$-representation.
\end{proof}

By the Plancherel formula, the map $\bigoplus_{(P,\sigma)} \pi_\sigma$ is injective on $\cS(G,F)$ (in fact, only the standard minimal parabolic subgroup is needed in the direct sum). This map is known as the Fourier transform, and an interesting problem is to characterize the Fourier image as functions on the various $\fa^*$ with particular properties. Such a characterization is known as a Paley-Wiener theorem. For $C_c^\infty(G)$, the Fourier image was first characterized by Arthur \cite{arthur}, and later characterized in a different way by Delorme \cite{delormePW} (both characterizations turn out to be the same \textit{a priori}; see \cite{vdbSouaifi}). 

We will make use of techniques that Delorme developed in \cite{delormePW}, and adapt these to $\cS(G)$. These techniques make use of several deep results in representation theory, including the theory of Knapp-Stein intertwining operators, Vogan's minimal $K$-types, and Vogan-Zuckerman classification. The results used are summarized in the first sections of \cite{delormePW} and \cite{delormeHC}. We list the definitions and theorems relevant to the exposition given here.

\begin{defn}
	Given a cuspidal pair $(P, \sigma)$, we write $\minK{\sigma}$ to denote the set of $K$-types appearing in $I_\sigma$ that have minimal length out of the $K$-types which appear in $I_\sigma$. Such $K$-types will be called the \textit{minimal $K$-types} for $\sigma$. We will use the notation $\|\sigma\|$ to denote the length of any element of $\minK{\sigma}$.
\end{defn}
\begin{thm}[{\cite[Theorem 1.1]{voganmainpaper}}]
	The elements of $\minK{\sigma}$ appear with multiplicity $1$ in $I_\sigma$.
\end{thm}

We will need, in particular, Vogan-Zuckerman classification on the unitary principal series. The statement we use is \cite[(1.7)]{delormePW}, but the reference is \cite[Chapter 6]{voganbook}. In the following theorem, we use $I^P_{\sigma,\lambda}$ to denote the $(\fg, K)$-module $(\pi^P_{\sigma,\lambda}, I_\sigma)$.
\begin{thm}\label{thm: voganclassification}
	Given a cuspidal pair $(P,\sigma)$ and $\lambda \in \overline{\fa^*_{P,+}}$, there is a unique decomposition (up to reordering)
	\begin{equation}\label{eq: voganclassification}
		I^P_{\sigma, \lambda} \cong I^P_{\sigma,\lambda}[\mu_1] \oplus \cdots \oplus I^P_{\sigma,\lambda}[\mu_l],
	\end{equation}
	where $\mu_i \in \minK{\sigma}$, and  $I^P_{\sigma,\lambda}[\mu_i]$ are subrepresentations with a unique quotient $J^P_{\sigma,\lambda}[\mu_i]$ containing $\mu_i$. In particular, every irreducible subquotient of $I^P_{\sigma,\lambda}$ contains a minimal $K$-type of $I_\sigma$. Moreover, if $(Q,\sigma)$ is another cuspidal pair with the same Levi subgroup as $P$, and if $\lambda \in \overline{\fa^*_{P,+}} \cap \overline{\fa^*_{Q,+}}$, then  $J_{\sigma,\lambda}^P[\mu] = J_{\sigma,\lambda}^Q[\mu]$.
\end{thm}

We will also make use of the Knapp-Stein intertwining operators, as well as a particular normalization of them.
\begin{thm}[See \cite{knappstein}]\label{thm: knappstein}
	Fix two parabolic subgroups $P = MAN_P$ and $Q = MAN_Q$ with the same Levi subgroup $MA$. Let $\sigma \in \widehat{M}_d$. If $\lambda \in \fa^*_{P,+}$ and $v \in I_\sigma$, the integral
	\[(A(Q,P,\sigma,\lambda)v)(u) = \int_{\theta({N}_P) \cap N_Q} a_P(\nbar)^{-\lambda-\rho}v(uk_P(\nbar^{-1})) d\nbar\]
	converges for each $u \in K$, and defines an element $A(Q,P,\sigma,\lambda)v \in I_\sigma$. Moreover, the map $v \mapsto A(Q,P,\sigma,\lambda)v$ defines an intertwining operator
	\[A(Q,P,\sigma,\lambda): (\pi^P_{\sigma,\lambda}, I_\sigma) \to (\pi^Q_{\sigma,\lambda}, I_\sigma).\]
	Finally, for each $v,w \in I_\sigma$, the map $\lambda \mapsto \langle A(Q,P,\sigma,\lambda)v, w\rangle$ extends to a meromorphic function in $\lambda$, and in this way we obtain intertwining operators $A(Q,P,\sigma,\lambda)$ for generic $\lambda \in \fa^*$.
\end{thm}

\begin{defn}
	The family of operators $A(Q,P,\sigma,\lambda)$ from Theorem \ref{thm: knappstein} is called the \textit{family of (unnormalized) Knapp-Stein intertwining operators}. When $\sigma$ is unambiguous we will write $A(Q,P,\lambda)$ instead of $A(Q,P,\sigma,\lambda)$.

	Fixing a minimal $K$-type $\mu_0 \in \minK{\sigma}$, $A(Q, P, \sigma, \lambda)$ acts by a scalar on $p_{\mu_0} I_{\sigma}$ (since $\mu_0$ has multiplicity $1$), which we denote by $c_{\mu_0}(Q, P, \sigma, \lambda)$. The \textit{normalized} intertwining operator is
	\[\cA(Q,P,\lambda) = \cA(Q, P,\sigma, \lambda) = c_{\mu_0}(Q,P, \sigma,\lambda)^{-1}A(Q,P,\sigma,\lambda).\]
\end{defn}
\begin{thm}[{See \cite[(1.12)]{delormePW}}]\label{thm: norminterwiner}
	The operator $\cA(Q,P,\lambda)$ unitary on $i\fa^*_0$, and is independent of $\lambda$ on minimal $K$-types. When $P,Q,R$ are parabolic subgroups with common Levi subgroup, we have
	\[\cA(R,Q,\lambda)\cA(Q,P,\lambda) = \cA(R,P,\lambda), \quad \cA(P,Q,\lambda) \cA(Q,P,\lambda)= \Id_{I_\sigma}\]
	as meromorphic functions of $\lambda$. Additionally, for each finite subset $F \subset \widehat{K}$, the operator $\cA(Q,P,\lambda)$ is rational in $\lambda$ when restricted to $p_F I_\sigma$.
\end{thm}

Let $W(\fg_0, \fa_0) =  W(A) = N_K(\fa_0)/Z_K(\fa_0)$ denote the (restricted) Weyl group corresponding to $A$. We will write $W_\sigma = W_\sigma(A)$ for the stabilizer of $\sigma \in \widehat{M}_d$ under the action
\[(w \cdot \sigma)(m) :=\sigma(w^{-1}m).\]

\begin{thm}[See \cite{voganbook}, \cite{knappstein}, {\cite[Theorem 1]{delormeHC}}]\label{thm: WandRGroups}
	There are subgroups $W_\sigma^0, R_\sigma$ of $W_\sigma$, and a simply transitive action of $\widehat{R}_\sigma$ on $\minK{\sigma}$ with the following properties:
	\begin{itemize}
		\item $W_\sigma^0$ is a normal subgroup of $W_\sigma$, generated by reflections on $\fa^*$, and
		\[W_\sigma = R_\sigma \ltimes W_\sigma^0.\]
		\item $R_\sigma$ is isomorphic to a direct product of copies of $\bZ/2\bZ$.
		\item Given $\lambda \in \overline{\fa^*_{P,+}}$ and $\mu, \nu \in \minK{\sigma}$, two subquotients $J^P_{\sigma,\lambda}[\mu], J^P_{\sigma,\lambda}[\nu]$ of $I_{\sigma,\lambda}^P$ (notation from Theorem \ref{thm: voganclassification}) are isomorphic if and only if $\mu$ and $\nu$ are related by an element of $\widehat{R}_{\sigma}(\lambda)$, the characters of $R_\sigma$ which vanish on the set $\{w \in R_\sigma: w \lambda \in W_\sigma^0\lambda\}$.
	\end{itemize}
	Moreover, the subgroup $W_\sigma^0$ and action of $\widehat{R}_\sigma$ on $\minK{\sigma}$ are uniquely determined by these properties (see \cite[Theorem 1 (v)]{delormeHC}).
\end{thm}
For the existence of these groups, see \cite[Lemma 4.3.14 and Theorem 4.4.8]{voganbook} (using \cite[Notation 6.6.3 and Theorem 6.6.15]{voganbook}).
\begin{defn}
	The subgroup $R_\sigma$ is known as the \textit{$R$-group}. We also characterize $W_\sigma^0$ by the (Knapp-Stein \cite{knappstein}) property that $\cA(P,w,\sigma,0)$ is the identity on $I_{\sigma}$ (see \cite[Theorem 1 (v)]{delormeHC}).
\end{defn}
Given $w \in W(A)$, define the map
\[T(w): I_\sigma \to I_{w\cdot\sigma}\]
by $(T(w)\varphi)(k) = \varphi(kw)$, which intertwines $\pi^P_{\sigma,\lambda}$ and $\pi^{wPw^{-1}}_{w\cdot \sigma, w\lambda}$. Now define
\[\cA(P,w, \lambda) = \cA(P,w,\sigma,\lambda) = T(w) \cA(w^{-1}Pw, P, \sigma,\lambda).\]
As before, $\cA(P,w, \lambda)$ acts as a scalar, independent of $\lambda$, on minimal $K$-types.
\begin{defn}\label{defn: hatr mu nu}
	Given $\mu \in \minK{\sigma}$, we write $a^\mu(w)$ to denote the (nonzero) scalar which defines the action of $\cA(P, w, \lambda)$ on $p_{\mu} I_{\sigma}$ (which is independent of $\lambda$). Given $\mu,\nu \in \minK{\sigma}$, the map
	\[w \mapsto a^\mu(w) (a^{\nu}(w))^{-1} : W_\sigma \to \bC^\times\]
	is a character of $W_\sigma$, trivial on $W_\sigma^0$, hence a character of $R_\sigma$ which we denote by $\hat{r}_{\mu\nu}$. 
\end{defn}
\begin{thm}[See {\cite[Theorem 1 (iv)]{delormeHC}}]\label{thm: hatr mu nu}
	For each $\mu, \nu \in \minK{\sigma}$, $\hat{r}_{\mu\nu}$ of Definition \ref{defn: hatr mu nu} is the unique element of $\widehat{R}_\sigma$ such that $\hat{r}_{\mu\nu} \cdot \nu = \mu$ under the action of $\widehat{R}_\sigma$ on $A(\sigma)$ given in Theorem \ref{thm: WandRGroups}.
\end{thm}

\subsection{Polynomial division of rapidly decreasing functions}\label{sec: polynomials}
We will need some lemmas regarding polynomial division of rapidly decreasing functions. The following lemmas are analogues of \cite[Lemma B.1, Theorem B.1]{delormeclozel2}, and \cite[Lemmas 7,8]{delormeclozel}.

We begin with a lemma due to Ehrenpreis \cite[Theorem 1.4]{ehrenpreisSCV}.
\begin{lem}[{\cite[Lemma B.1 (i)]{delormeclozel2}}]\label{lem: ehrenlemma}
	Given a complex polynomial $p$ on $\bC^n$, there exist constants $c, m \geq 0$ such that if $h$ is holomorphic on the closed polydisc $\Delta_\rho(z_0) = \{z \in \bC^n: \|z - z_0\| \leq \rho\}$ (where $\|z\| = \sup_{j} |z_j|$), then
	\[|h(z_0)| \leq c \rho^{-m}\sup_{z \in \Delta_\rho(z_0)}|p(z) h(z)|.\]
\end{lem}

We say that a nonzero polynomial $p$ \textit{divides} a holomorphic function $f$ if $f/p$ extends to a holomorphic function.
\begin{lem}\label{lem: polydivision}
	Let $V_0$ be a Euclidean vector space with complexification $V$.
	Let $f \in \PW(V)$ and $p \in \bC[V]$. If $p$ divides $f$, then $f/p \in \PW(V)$.
\end{lem}
\begin{proof}
	Compare to \cite[Lemma B.1]{delormeclozel2}. We identify $V$ with $\bC^n$ via the inner product.	Set $h = f/p$, which extends to an entire function on $\bC^n$. Setting $\rho = 1$ in Lemma \ref{lem: ehrenlemma}, for each $N$ and $k \geq 0$,
	\begin{align*}
		\sup_{\|\Re z\| \leq k} (1+\|z\|)^N|h(z)| \leq c \sup_{\|\Re z\| \leq k} (1+\|z\|)^N  \sup_{z' \in \Delta_1(z)}|f(z')|.
	\end{align*}
	When $z' \in \Delta_1(z)$, then $1 + \|z\| \leq 2(1+\|z'\|)$ and
	$\|\Re z'\| \leq \|\Re z\| + 1$ by the triangle inequality. Therefore,
	\begin{align}\label{eq: polydiv estimate}
		\sup_{\|\Re z\| \leq k} (1+\|z\|)^N|h(z)| &\leq 2^N c \sup_{\|\Re z\| \leq k + 1}  (1+\|z\|)^N |f(z)| < \infty.
	\end{align}
	This proves that $h \in \PW(V)$.
\end{proof}

\begin{thm} \label{thm: Rais Theorem}
	Fix a Euclidean vector space $V_0$ with complexification $V$, and a finite group $W$ generated by reflections on $V_0$. There exist homogeneous complex polynomials $p_i \in \bC[V]$ such that
	\[\PW(V) = \sum_i p_i \PW(V)^W\]
	Moreover, the above sum is free in the sense that the decomposition $f = \sum p_i f_i$ (where $f \in \PW(V), f_i \in \PW(V)^W$) is unique.
\end{thm}
This analogue of this theorem for compactly supported smooth functions is due to Rais \cite{rais}. See also \cite[Theorem B.1]{delormeclozel2}; the proof of Theorem \ref{thm: Rais Theorem} is almost identical, and so we shall not provide complete details.
\begin{proof}[Proof sketch]
	According to \cite[Lemma 8]{HCinvariants}, $\bC[V]$ is free over $\bC[V]^W$ with homogeneous basis $p_1, \ldots, p_{|W|} \in \bC[V]$. 
	
	Write $W = \{w_1, \ldots, w_{|W|}\}$. Given $f \in \PW(V)$, for each $\lambda \in \bC$ we obtain a $|W| \times |W|$ linear system
	\[f(w_j\lambda) = \sum_{i=1}^{|W|} p_i(w_j\lambda) f_i(\lambda), \quad j = 1, \ldots, |W|.\]
	If $D(\lambda) = \det([p_i(w_j \lambda)])$ denotes the determinant of this system, then $D(\lambda)$ is a polynomial which is nonzero because of the polynomial case above. By Cramer's rule, we can find (unique) functions $g_i(\lambda) \in \PW(V)$ such that
	\[D(\lambda )f(w_j\lambda) = \sum_{i} p_i(w_j\lambda)g_i(\lambda).\]
	Now, approximating $f$ by truncated Taylor series polynomials $q_n$, there are unique $W$-invariant polynomials $f_{n,i}$ such that $q_n = \sum p_i f_{n,i}$. Moreover, $g_{n,i}(\lambda) = D(\lambda) f_{n,i}(\lambda)$ converge to $g_i$ uniformly on compact sets as $n \to \infty$. The estimate \eqref{eq: polydiv estimate} (applied to $f = g_{n,i}$ and $F = f_{n,i}$) is then used to prove that $D$ divides each $g_i$, and that $f_{n,i}$ converge to $f_i=g_i/D$ uniformly on compact sets. By Lemma \ref{lem: polydivision}, we have $f_i \in \PW(V)^W$ . Uniqueness of the $f_i$ follows from uniqueness of $g_i$.
\end{proof}

Let $P = MAN$ be a cuspidal parabolic subgroup. Let $T \subset M$ be a $\theta$-stable Cartan subgroup contained in $K$, with Lie algebra $\ft_0$. We set $\fh = \fa + \ft$. We need a particular action of the restricted Weyl group $W(\fg_0, \fa_0)$ on $i\ft_0^*$ due to Knapp. Its important property is its relationship to the action of $W(\fg_0, \fa_0)$ on $\widehat{M}_d$.
\begin{thm}[{\cite[Theorem 3.7, Theorem 4.10]{knappcommutativity}}]\label{thm: knappweylaction}
	Let $M_0$ denote the connected component of $M$ at the identity. There exists an action of $W(\fg_0, \fa_0)$ on $i\ft^*_0$ such that, given $\sigma \in \widehat{M}_d$ and Harish-Chandra parameter $\Lambda_\sigma$ of $\sigma\vert_{M_0}$ (recall Definition \ref{defn: HC parameter of sigma}), if $(w\cdot \sigma)\vert_{M_0}  \cong \sigma\vert_{M_0}$ then $w \cdot \Lambda_\sigma = \Lambda_\sigma$. For each $w \in W(\fg_0, \fa_0)$, there exists a representative $k \in N_K(\fa_0)$ of $w$ such that $w$ acts by $\Ad_k$ on $i\ft_0$.
\end{thm}  

See also the discussion before \cite[Lemma 7]{delormeclozel}. 
Note that, in particular, $W_\sigma$ acts trivially on $\Lambda_\sigma$. Theorem \ref{thm: knappweylaction} allows us to define an action of $W(\fg_0, \fa_0)$ (in particular, $W_\sigma^0$) on $\fh^*$. Moreover, because this action comes from $\Ad_k$ for some $k \in K$, for each $w \in W(\fg_0, \fa_0)$ there exists some $w' \in W(\fg, \fh)$ such that $w$ acts on $\fh^*$ the same way as $w'$. 
\begin{lem}\label{lem: hcparamrestrictionthing}
	Fix $\sigma \in \widehat{M}_d$. Let $\Lambda_\sigma \in i\ft_0^*$ be the Harish-Chandra parameter of $\sigma\vert_{M_0}$ (recall Definition \ref{defn: HC parameter of sigma}) Then
	\[\PW(\fa^*)^{W_\sigma^0} = \{\lambda \mapsto f(\Lambda_\sigma + \lambda): f \in \bC[\fa^*]^{W_\sigma^0}\PW(\fh^*)^{W(\fg,\fh)}\}.
	\]
\end{lem}
\begin{proof}
	Compare to \cite[Lemmas 7,8]{delormeclozel}. We first note that
	\begin{equation}\label{eq: hcparamrestrictionproofeq1}
		\{\lambda \mapsto f(\Lambda_\sigma + \lambda): f \in \PW(\fh^*)^{W_\sigma^0}\} = \PW(\fa^*)^{W_\sigma^0}.
	\end{equation}
	Indeed, choose any $W_\sigma^0$-invariant $F \in \PW(\ft^*)$ such that $F(\Lambda_\sigma) = 1$. Then any $f \in \PW(\fa^*)^{W_\sigma^0}$ is the restriction of $(\lambda_1 \oplus \lambda) \mapsto F(\lambda_1) f(\lambda)$ to $\Lambda_\sigma + \fa^*$, where $\lambda_1 \in \ft^*, \lambda \in \fa^*$.
	
	Next, by Theorem \ref{thm: Rais Theorem} we see that
	\[\PW(\fh^*) = \bC[\fh^*] \PW(\fh^*)^{W(\fg, \fh)}.\]
	Now note that elements of $\PW(\fh^*)^{W(\fg, \fh)}$ are also $W_\sigma$-invariant by the discussion before this lemma. Therefore, averaging the decomposition $f = \sum p_i f_i$ by the action of $W_\sigma^0$, we have 
	\begin{equation}\label{eq: hcparamrestrictionproofeq2}	
		\PW(\fh^*)^{W_\sigma^0} = \bC[\fa^*]^{W_\sigma^0}  \PW(\fh^*)^{W(\fg, \fh)}.
	\end{equation}
	The lemma follows from \eqref{eq: hcparamrestrictionproofeq1} and \eqref{eq: hcparamrestrictionproofeq2}.
\end{proof}

\section{Filtrations on $\cS(G,F)$ and $C_r^*(G,F)$}
\begin{defn}
	We say two cuspidal pairs $(P = L_P N_P, \sigma), (Q = L_Q N_Q, \tau)$ are \textit{$G$-conjugate} if there is an element of $G$ which conjugates $L_P$ to $L_Q$ and conjugates $\sigma$ to $\tau$. Note that a $G$-conjugacy class of cuspidal pairs is also known as an associate class (see \cite[Definition 5.2]{clarecrisphigson}).
	
	Because a given $G$-conjugacy class of cuspidal pairs $[P,\sigma]$ only depends on the Levi subgroup $L$, we may also write $[L, \sigma]$. Moreover, because the Levi subgroup is implicitly specified by the representation $\sigma$, we may simply write $[\sigma]$ for such a class.
\end{defn}
\begin{thm}[{See \cite[Theorem 7.17]{voganmainpaper}}]\label{thm: voganthingo}
	The sets $\minK{\sigma}$ partition $\widehat{K}$, and two such sets $\minK{\sigma}, \minK{\tau}$ are equal if and only if $[\sigma] = [\tau]$. 
\end{thm}
In particular, we see that the value $\|\sigma\|$ is independent of the cuspidal pair $(P, \sigma)$ up to $G$-conjugacy. We fix a total order on $G$-conjugacy classes $[\sigma]$ such that if $\|\sigma\| < \|\tau\|$ then $[\sigma] < [\tau]$. We fix representatives $(P_n, \sigma_n)$ for each $G$-conjugacy class, so that
\begin{equation*}\label{eq: totalordering}
	[\sigma_1] < [\sigma_2] < \cdots.
\end{equation*}
Let $\pi_n$ denote $\pi_{\sigma_n}$, and $p_n$ denote $p_{\minK{\sigma_n}}$. Often, $n$ will be fixed and we will write $\sigma = \sigma_n$.

\begin{defn}
	For each $n \in \bN$, define the ideal $\cJ_n \subset \cS(G,F)$ by
	\[\cJ_n = \bigcap_{m>n} \ker \left(\pi_m: \cS(G,F) \to \PW(\fa^*, \End(p_F I_{\sigma_m})) \right).\] We define ideals $J_n \subset C_r^*(G,F)$ analogously. We also set $\cJ_0 = J_0 = 0$.
\end{defn}
By Theorem \ref{thm: voganthingo}, for each finite set $F \subset \Khat$, we have $p_F \cJ_n p_F =\cS(G,F)$ and $p_F J_n p_F = C_r^*(G,F)$ for large enough $n$.

Fix $R \geq 0$ and set $F = \{\gamma \in \Khat: \|\gamma\| \leq R\}$. For each $n \in \bN$, we have the injection
\begin{equation*}\label{eq: subquo inj}
	\overline{\pi}_{n}:  \cJ_n/\cJ_{n-1} \hookrightarrow \pi_n(\cS(G,F)).
\end{equation*}
Therefore, as an algebra we may identify $\cJ_n / \cJ_{n-1}$ with the image of $\cJ_n$ under $\pi_n$. Similarly, we may identify $J_n/J_{n-1}$ with $\pi_n(J_n)$.

The main theorem will be a consequence of the following four theorems describing these subquotients. We let $w \in W_{\sigma}$ act on $\PW(\fa^*, \End(p_{A(\sigma)}I_\sigma))$ by
\[(w \cdot \phi)(\lambda) := \cA(P,w,w^{-1}\lambda)\phi(w^{-1}\lambda)\cA(P,w^{-1},\lambda).\]
The above also defines an action of $W_\sigma$ on $C_0(i\fa^*_0, \End(p_{A(\sigma)}I_\sigma))$.

\begin{thm}\label{thm: Jsig minK}
	For each $n \in \bN$ with $A(\sigma_n) \subset F$,
	\[(\cJ_n/\cJ_{n-1}) p_{n}(\cJ_n/\cJ_{n-1}) = \cJ_n/\cJ_{n-1}.\]
\end{thm}

\begin{thm}\label{thm: Cstar mink}
	For each $n \in \bN$ with $A(\sigma_n) \subset F$
	\[\overline{(J_n/J_{n-1})p_{n}(J_n/J_{n-1})} = J_n/J_{n-1}.\] 
\end{thm}
\begin{thm}\label{thm: PW on lowktypes}
	For each $n \in \bN$ with $\minK{\sigma_n} \subset F$,
	\[\pi_n(p_n\cJ_n p_n) = \PW(\fa^*,\End(p_n I_{\sigma_n}))^{W_{\sigma_n}}.\]
\end{thm}

\begin{thm}\label{thm: Cstar on lowktypes}
	For each $n \in \bN$ with $\minK{\sigma_n} \subset F$,
	\[\pi_n(p_nJ_np_n) = C_0(i\fa^*_0,\End(p_nI_{\sigma_n}))^{W_{\sigma_n}}.\]
\end{thm}

Theorem \ref{thm: Jsig minK} is the most technical, and is a key consequence of Delorme's techniques. We will leave this theorem to Section \ref{sec: deferredproof}, and prove the other three theorems here.

We need to rewrite Theorem \ref{thm: PW on lowktypes}  in a form compatible with \cite[(1.38)]{delormePW}. This will also allow us to understand the action of $W_\sigma$ on $\PW(\fa^*, \End(p_n I_\sigma))$. Recall from Theorem \ref{thm: WandRGroups} that $W_\sigma$ decomposes as a semidirect product
\[W_\sigma = R_\sigma W_\sigma^0,\]
where $W_\sigma^0$ is characterized by the property that $\cA(P_n, w, \lambda)$ is the identity, and (among other properties) $R_\sigma$ is a product of copies of $\bZ/2$. Moreover, the characters $\widehat{R}_\sigma$ act on $\minK{\sigma}$ simply transitively. We also recall from Definition \ref{defn: hatr mu nu} and Theorem \ref{thm: hatr mu nu} that the intertwining operators $\cA(P,w,\lambda)$ acts by scalars $a^{\mu}(w)$ on $p_{\mu} I_\sigma$ for each $\mu \in \minK{\sigma}$, and $\hat{r}_{\mu\nu}(w) := a^{\mu}(w)(a^{\nu}(w))^{-1}$ is a character of $R_\sigma$, and moreover is the unique element of $\widehat{R}_\sigma$ such that $\hat{r}_{\mu\nu} \cdot \nu = \mu$.

By the above, each $w = w^0 r\in W_\sigma$ (where $w^0 \in W_\sigma^0$, $r \in R_\sigma$) acts on $f \in p_{\mu} \PW(\fa^*, \End(p_{A(\sigma)}I_\sigma))p_{\nu}$ by
\[(w \cdot f)(\lambda) = \hat{r}_{\mu\nu}(r)f((w^0)^{-1} \lambda).\]
Therefore, if $\PW(\fa)^{W_\sigma^0}(\hat{r}_{\mu\nu})$ denotes the space of $W_\sigma^0$-invariant functions $f$ such that $f(r\lambda) = \hat{r}_{\mu\nu}(r) f(\lambda)$ for each $r \in R_\sigma$, then
\begin{equation}\label{eq: minKfulldecomp}
	p_{\mu} \PW(\fa^*, \End(p_{\minK{\sigma}}I_\sigma))^{W_\sigma} p_\nu = \PW(\fa^*)^{W_\sigma^0}(\hat{r}_{\mu\nu}) \otimes \Hom(p_\nu I_\sigma, p_\mu I_\sigma).
\end{equation}
In particular, Theorem \ref{thm: PW on lowktypes} is equivalent to the identification
\begin{equation}\label{eq: minkKfulldecompactual}
	p_\mu \pi_n(J)_n p_\nu =\PW(\fa^*)^{W_\sigma^0}(\hat{r}_{\mu\nu}) \otimes \Hom(p_\nu I_\sigma, p_\mu I_\sigma).
\end{equation}
Our discussion also proves that the action of $W_\sigma$ on $\PW(\fa^*, \End(p_n I_\sigma))$ is induced by its action on $\fa^*$ and a diagonal action on $\End(p_n I_\sigma))$ (with respect to the entries).

\begin{proof}[Proof of Theorem \ref{thm: PW on lowktypes}]	
	The proof of \eqref{eq: minkKfulldecompactual} is identical to that of \cite[(1.38)]{delormePW} when adapted to $\cS(G)$. To adapt the proof to $\cS(G)$, we use \cite[Theorem 3]{delormeflensted} instead of \cite[Theorem 2]{delormeflensted}, and we use Lemmas \ref{lem: polydivision} and \ref{lem: hcparamrestrictionthing} in place of \cite[(1.21), (1.37)]{delormePW}. As \eqref{eq: minkKfulldecompactual} is equivalent to the theorem, we are done.
\end{proof}

To prove Theorems  \ref{thm: Cstar mink} and \ref{thm: Cstar on lowktypes}, we need the following. Recall that $(\pi^P_{\sigma,\lambda}, \cH_\sigma)$ denotes the (principal series) Hilbert space representation of $G$ whose $(\fg, K)$-module is $(\pi^P_{\sigma,\lambda}, I_\sigma)$.

\begin{thm}[See {\cite[Propositions 5.17, 6.7 and Theorem 6.8]{clarecrisphigson}}]\label{thm: cstar subquo}
	Define $\pi_n: C_r^*(G) \to C_0(i\fa^*_0, \cB(\cH_\sigma))$ by $\pi_n(\phi)(\lambda) = \pi^{P_n}_{\sigma_n,\lambda}(\phi)$. The $C^*$-algebra homomorphism
	\[\bigoplus_{n\in \bN} \pi_n: C_r^*(G) \to \bigoplus_{n \in \bN} \pi_n(C_r^*(G))\]
	is an isomorphism. Consequently, the map $J_n/J_{n-1} \to \pi_n(C_r^*(G,F))$ is an isomorphism of $C^*$-algebras. Moreover, for each cuspidal pair $(P, \sigma)$,
	\begin{equation}\label{eq: CCHtheory}
		\pi_{\sigma}(C_r^*(G)) = C_0(i\fa^*_0, \cK(\cH_{\sigma}))^{W_\sigma},
	\end{equation}
	where $\cK(\cH_\sigma)$ denotes compact operators on $\cH_\sigma$, and each $w \in W_\sigma$ acts on $f \in C_0(i\fa^*_0, \cK(\cH_\sigma))$ via $(w \cdot f)(\lambda) = \cA(P, w, w^{-1}\lambda)f(w^{-1}\lambda)\cA(P, w^{-1}, \lambda)$.
\end{thm}
Theorem \ref{thm: Cstar on lowktypes} follows by multiplying $p_n$ to the left and right in \eqref{eq: CCHtheory}.

\begin{proof}[Proof of Theorem \ref{thm: Cstar mink}]
	It suffices to prove that
	\[\cI := \pi_n(C_r^*(G)p_{n}C_r^*(G))\] is dense in $\pi_n(C_r^*(G))$, because then $p_F \cI p_F = \pi_{n}((J_n/J_{n-1})p_{n}(J_n/J_{n-1}))$ is dense in $p_F\pi_n( C_r^*(G))p_F = \pi_n(J_n/J_{n-1})$. 
	
	Suppose $\cI$ is not dense in $\pi_n(C_r^*(G))$. Then there is an irreducible representation of $\pi_n(C_r^*(G))$ which vanishes on $\cI$ (this is a consequence of \cite[Proposition 2.11.2 (i)]{dixmiercstar}).  
	
	By Theorem \ref{thm: cstar subquo}, every irreducible representation of $\pi_n(C_r^*(G))$ is an irreducible subquotient of the representation $(\pi^{P_n}_{{\sigma_n},\lambda},\cH_{\sigma_n})$ for some $\lambda \in i\fa^*_0$. Vogan's classification \eqref{eq: voganclassification} implies that a minimal $K$-type must be contained in this subquotient, and therefore this representation does not vanish on $\cI$, proving the claim.
\end{proof}

\section{Proof of Theorem \ref{thm: mainthm}}

We now turn to the proof of Theorem \ref{thm: mainthm}, which states that $\cS(G,F) \to C_r^*(G,F)$ induces an isomorphism in $K$-theory. It suffices to prove that the mapping cone of this inclusion vanishes in $K$-theory. The proof begins by studying the inclusion maps $\cJ_n/\cJ_{n-1} \to J_n/J_{n-1}$.

We will apply the results of Section \ref{sec: ktheory} regarding Fr{\'e}chet algebra $K$-theory and Morita equivalence. In order to do so, we must first complete the space $\cJ_n/\cJ_{n-1}$ to suitable Banach algebras. We recall that the topology of $\cS(G)$ is generated by the seminorms
\[\|\phi\|_{\cS(G),N,k} = \sum_{|I|, |J| \leq k} \int_G (1+\|g\|)^N|L_{X^I} R_{X^J}\phi(g)| dg,\]
for some fixed choice of orthonormal basis $X_1, \ldots, X_{\dim G}$ of $\fg$. Now, the action of $K$ is continuous with respect to the seminorms $\|\cdot\|_{\cS(G),N,k}$, and the kernel of $\|\cdot\|_{\cS(G),N,k}$ is a $K$-submodule of $\cS(G)$. Therefore, $p_F$ extends to a multiplier of the corresponding completions $\cS_{N,k}(G)$. We may also complete $\cJ_n/\cJ_{n-1}$ with respect to the corresponding (sub)quotient seminorms, and $p_F$ extends to these completions too. In particular, Theorem \ref{thm: Jsig minK} implies that $p_F$ is a full idempotent of $\cJ_n/\cJ_{n-1}$.

On the Fourier transform side, recall that we also have norms on $\PW(\fa^*, V)$ for a fixed finite-dimensional normed space $V$. Given $k > 0$ and each $N \in \bN \cup \{0\}$, we set $X = \{\lambda \in \fa^*: | \Re \lambda| \leq k\}$ and
\begin{equation}\label{eq: norm on PW}
	\|f\|_{N,k} = \|f\|_{N,X} = \sup_{\lambda \in X} (1+|\lambda|)^N \|f(\lambda)\|_V. 
\end{equation}

These norms generate the topology of $\PW(\fa^*, V)$.
\begin{lem}\label{lem: banachcompletePW}
	Let $W$ be a finite subgroup of the orthogonal group $O(\fa^*_0)$, and suppose the action of $W$ on $\PW(\fa^*,V)$ has the form
	\[(w \cdot f)(\lambda)= D(w)f(w^{-1}\lambda),\]
	where $D(w) \in \Aut(V)$ is independent of $\lambda$. Then the completion of $\PW(\fa^*,V)^W$ with respect to the norm $\|\cdot\|_{N,X}$ defined by \eqref{eq: norm on PW} is isomorphic to
	\begin{align*}
		\PW_N(X,V)^W &= \{f: X \to V: \text{\emph{ $(1+|\lambda|)^Nf \in C_0(X)$, $D(w)f(w^{-1}\lambda) = f(\lambda)$,}} \\
		&\text{\emph{$f$ is continuous on $X$, holomorphic on the interior of $X$}}\}.
	\end{align*}
\end{lem}
The above lemma is a consequence of the next two lemmas.
\begin{lem}\label{lem: banachcompletePWsub1}
	With notation as in Lemma \ref{lem: banachcompletePW}, for each $N \geq 0$, the space $\bigcap_{M=0}^\infty \PW(X, V)^W$ is dense in $\PW_N(X,V)^W$ with respect to $\|\cdot\|_{N,X}$.
\end{lem}

\begin{proof}
	Fix an orthonormal basis $\{e_i\}$ on $\fa^*_0$ and set $\lambda^2 = \sum_i \langle \lambda, e_i\rangle^2 \in \bC$. That is, if we identify $\fa^*$ with $\bC^n$ and write $\lambda = (\lambda_1, \ldots, \lambda_n)$, then $\lambda^2 = \sum_i \lambda_i^2$. This quantity is independent of the basis chosen. Note that $\lambda^2 \leq 0$ when $\lambda \in i\fa^*_0$, and that $(w \cdot \lambda)^2 = \lambda^2$ for $w \in W$.
	
	Suppose $f \in \PW_N(X,V)^W$. Now, $g_s(\lambda) = e^{s\lambda^2}f(\lambda)$ is in $\bigcap_M\PW_M(X,V)^W$ for each $s > 0$. For any $\varepsilon > 0$ and any compact set $Y\subset X$ we may choose $s > 0$ such that $|e^{s\lambda^2} -1| < \varepsilon$ for $\lambda \in Y$, so that 
	\[(1+|\lambda|)^N\|g_s(\lambda)- f(\lambda)\|_V < \|f\|_{N,X} \varepsilon\]
	for $\lambda \in Y$. If we choose $Y$ large enough so that $(1+|\lambda|)^N \|f(\lambda)\|_V < \varepsilon$ when $\lambda \notin Y$, then we see that $g_s$ approximates $f$ in $\|\cdot\|_{N,X}$. This proves the lemma.
\end{proof}
\begin{lem}\label{lem: banachcompletePWsub2}
	With notation as in Lemma \ref{lem: banachcompletePW}, the space $\PW(\fa^*, V)^W$ is dense in $\bigcap_{M=0}^\infty \PW_M(X,V)^W$ with respect to $\|\cdot\|_{N,X}$ for each $N \geq 0$.
\end{lem}
\begin{proof}
	As in the previous lemma, we fix an orthonormal basis $\{e_i\}$ on $\fa^*_0$ and set $\lambda^2 = \sum_i \langle \lambda, e_i\rangle^2$.
	
	Fix $f \in \bigcap_M\PW_M(X,V)^W$. We define, for each $t > 0$, $\varphi_t(\lambda) = (\pi t)^{(\dim \fa)/2}e^{t\lambda^2}$ and
	\begin{equation}\label{eq: mollifiedPW}
		h_t(\lambda) = \int_{i\fa^*_0} f(x) \varphi_t(\lambda - x)dx.	
	\end{equation}
	Note that $\int_{i\fa^*} \varphi_t(x) dx = 1$ for each $t >0$.
	
	First, we prove that $h_t \in \PW(\fa^*, V)$. Indeed, for $M > 0$, via $(1+|\lambda|)^M \leq (1+|\lambda - x|)^M(1+|x|)^M$,
	\[\int_{i\fa^*_0} (1+|\lambda|)^M \|f(x)\|_V |\varphi_t(\lambda - x)|dx \leq C(\lambda)\|f\|_{N,X} \|\varphi\|_{N+2\dim\fa,|\Re \lambda|} < \infty,\]
	where $C(\lambda) = \int_{i\fa^*_0}(1+|\lambda - x|)^{-2\dim\fa}dx$, which is bounded in $\lambda$ when $|\Re \lambda|$ is bounded. It also follows that $h_t$ is holomorphic on $\fa^*$ because we can pass the derivative through the integral. This proves that $h_t \in \PW(\fa^*, V)$.
	
	Now we prove that $h_t$ approximates $f$. Fix $\varepsilon \in (0,1)$ and choose $\delta \in (0, \varepsilon)$ such that
	\[\|(1+|\lambda|)^Nf(z) - (1+|\lambda - x|)^Nf(\lambda -x)\|_V < \varepsilon\]
	whenever $|x| < \delta$, $x \in i\bR$, and $\lambda \in X$. Then, for each $\lambda \in X$ and $|x| < \delta$, by noting $(1+|\lambda|)^N = ((1+|\lambda-x|)+(|\lambda| - |\lambda-x|))^N$ and applying binomial expansion, we get
	\[(1+|\lambda|)^N\|f(\lambda)-f(\lambda-x)\|_V < (1+2^N \|f\|_{N,X})\varepsilon.\]

	We will prove an estimate for $(1+|\lambda|)^N\|f(\lambda) - h_t(\lambda)\|_V$ for $\lambda$ in the interior of $X$. In this case, the integrand $f(x)\varphi_t(\lambda - x)$ is a holomorphic function of $x$ on a domain containing $s\lambda + i\fa^*_0$ for each $s \in [0,1]$. Therefore, we may shift the contour, so that
	\[\int_{i\fa^*_0} f(x) \varphi_t(\lambda - x)dx = \int_{\lambda+i\fa^*_0}f(x)\varphi_t(\lambda-x)dx = \int_{i\fa^*_0}f(\lambda-x)\varphi_t(x)dx.\]
	
	Our previous estimate proves
	\[\int_{\substack{x \in i\bR,\\|x| < \delta} } (1+|\lambda|)^N \|f(\lambda) - f(\lambda - x)\|_V|\varphi_t(x)| dx < (1+2^N \|f\|_{N,X})\varepsilon.\]
	Now, for large $t  > 1$, the fact that $(1+|x|)^N\varphi_1(x)$ is integrable implies
	\[\int_{\substack{x \in i\bR,\\|x| >\delta} } (1+|x|)^N |\varphi_t(x)|dx \leq\int_{\substack{x \in i\bR,\\|x| > t\delta} }(1+|x|)^N |\varphi_1(x)|dx < \varepsilon.\]
	Therefore, using $(1+|\lambda|) \leq (1+|\lambda -x|)(1+|x|)$,
	\[\int_{\substack{x \in i\bR,\\|x| >\delta} }(1+|\lambda|)^N \|f(\lambda) - f(\lambda - x)\|_V|\varphi_t(x)| < 2\|f\|_{N,X} \varepsilon.\]
	We have shown, for large $t$,
	\[\|(1+|\lambda|)^N (f - h_t)(\lambda)\|_{N,X} < (1+2^N\|f\|_{N,X})\varepsilon.\]
	
	The lemma follows by averaging $h_t$ with respect to the action of $W$.
\end{proof}
Lemma \ref{lem: banachcompletePW} follows immediately from Lemmas \ref{lem: banachcompletePWsub1} and \ref{lem: banachcompletePWsub2}.

In what follows, we will apply the above lemma to $V = \End(p_n I_\sigma)$. We have seen that $W_\sigma$ acts on $\PW(\fa^*, \End(p_n I_\sigma))$ as in the lemma.

\begin{prop}
	The inclusion of mapping cones
	\begin{align*}
		\MC(p_n(\cJ_n/\cJ_{n-1}) p_n \to p_n &(J_n/J_{n-1})p_n) \\&\hookrightarrow \MC(\cJ_n/\cJ_{n-1} \to J_n/J_{n-1})
	\end{align*}
	induces an isomorphism in $K$-theory. 
\end{prop}
\begin{proof}
	By Theorems \ref{thm: Jsig minK} and \ref{thm: cstar subquo}, together with Theorem \ref{thm: frechetmorita}, the inclusions $p_n(\cJ_n/\cJ_{n-1})p_n \hookrightarrow \cJ_n$ and $p_n(J_n/J_{n-1})p_n \hookrightarrow J_n$ induce isomorphisms in $K$-theory. 
	
	Now write 
	\begin{align*}
		\MC &= \MC(\cJ_n/\cJ_{n-1} \to J_n/J_{n-1}),\\	\MC_{\text{reduced}} &=\MC(p_n(\cJ_n/\cJ_{n-1}) p_n \to p_n (J_n/J_{n-1})p_n).
	\end{align*}
	
	If we apply the $6$-term exact sequence appearing in the proof of Lemma \ref{lem: mappingconething} to $\MC$ and $\MC_{\text{reduced}}$, we obtain a morphism of exact sequences
	\[\small \xymatrix@C=1em{
		\cdots  \ar[r]& K_i(p_n J_n/J_{n-1}p_n ) \ar[d]^{\cong} \ar[r]& K_{i+1}(\MC_{\text{reduced}}) \ar[d]\ar[r]& K_{i+1}(p_n \cJ/\cJ_{n-1}p_n) \ar[d]^{\cong} \ar[r]& \cdots\\
		\cdots \ar[r]& K_i(J_n/J_{n-1}) \ar[r]& K_{i+1}(\MC) \ar[r]& K_{i+1}(\cJ/\cJ_{n-1})  \ar[r]& \cdots.}\]
	The proposition follows from the five-lemma.
\end{proof}

\begin{thm}\label{thm: subquotient Kthy iso}
	The map $\cJ_n/\cJ_{n-1} \to J_n/J_{n-1}$ induces an isomorphism in $K$-theory.
\end{thm}
\begin{proof}	
	Set $\sigma = \sigma_n$. We assume $A(\sigma_n) \subset F$, as otherwise (by our assumptions on $F$) $\cJ_n/\cJ_{n-1} = 0$ and $J_n/J_{n-1} = 0$. By the previous proposition and Theorem \ref{thm: frechetmorita}, it suffices to show that the mapping cone of $p_n(\cJ_n/\cJ_{n-1}) p_n \to p_n(J_n/J_{n-1})p_n$ has zero $K$-theory. By Theorems \ref{thm: PW on lowktypes} and \ref{thm: Cstar mink}, we must show that the restriction map
	\begin{equation}\label{eq: PW to Cstar}
		\PW(\fa^*,\End(p_n I_\sigma))^{W_\sigma} \to C_0(i\fa^*_0, \End(p_n I_\sigma))^{W_\sigma}
	\end{equation}
	induces an isomorphism in $K$-theory.	
	
	We consider tubes $X$ of the form $\{\lambda \in \fa^*: \|\Re \lambda\| < k\}$ for some $k > 0$. If we define $	\PW_N(X,\End(p_n I_\sigma))^{W_\sigma}$ as in Lemma \ref{lem: banachcompletePW}, we have
	\[\PW(\fa^*,\End(p_n I_\sigma))^{W_\sigma} = \varprojlim_{X,N} 
	\PW_N(X,\End(p_n I_\sigma))^{W_\sigma}.\]
	Also, for any fixed tube $X$, we have
	\begin{equation}\label{eq: Cstar direct limit of PW}
		C_0(i\fa^*_0,\End(p_n I_\sigma))^{W_\sigma} = \varinjlim_l \PW_0(2^{-l}X,\End(p_n I_\sigma))^{W_\sigma},
	\end{equation}
	where the direct limit is in the category of Banach algebras and contractive morphisms (note that $\PW_0(X, \End(p_N I_\sigma))^{W_\sigma}$ consists of $W_\sigma$-invariant $C_0$ functions on $X$ which are holomorphic on the interior of $X$).
	
	We first show that $\PW_N(X,\End(p_n I_\sigma))^{W_\sigma}$ is independent of $N$ up to $K$-theory isomorphism. Indeed, the inclusion maps
	\[\PW_{N+1}(X,\End(p_n I_\sigma))^{W_\sigma} \hookrightarrow  \PW_N(X,\End(p_n I_\sigma))^{W_\sigma}\]
	have dense range, and we claim that the image is holomorphically stable. Fix $f \in\PW_{N+1}(X,\End(p_n I_\sigma))^{W_\sigma}$ and suppose it has a quasi-inverse $h \in \PW_{N}(X,\End(p_n I_\sigma))^{W_\sigma}$. Then for each $\lambda \in X$, the operator $1+f(\lambda)$ is invertible and
	\[h(\lambda)= -f(\lambda)(1+f(\lambda))^{-1}.\]	
	Now, as $f$ vanishes at infinity (on $X$), $\det(1+f)$ is bounded away from $0$. Cramer's rule implies that $(1+f)^{-1}$ is bounded on $X$, hence the function $-[(1+|\lambda|)^{N+1}f](1+f)^{-1}$ is bounded on $X$. Hence, $h \in \PW_{N+1}(X,\End(p_n I_\sigma))^{W_\sigma}$. This proves holomorphic stability, and Karoubi density implies that they have the same $K$-theory.
	
	We now show that $\PW_0(X, \End(p_n I_\sigma))^{W_\sigma}$ is independent of $X$ up to homotopy of Banach algebras. More precisely, we claim that the restriction map 
	\[\rest: \PW_0(X, \End(p_n I_\sigma))^{W_\sigma} \to \PW_0(X/2, \End(p_n I_\sigma))^{W_\sigma}\]
	has homotopy inverse
	\[\alpha_1: \PW_0(X/2, \End(p_n I_\sigma))^{W_\sigma} \to \PW_0(X, \End(p_n I_\sigma))^{W_\sigma},\]
	given by $(\alpha_1 f)(\lambda) = f(\lambda/2)$.
	The composition $\alpha_1 \circ \rest$ is the restriction to $t=1$ of the map
	\[\alpha_\bullet \circ \rest: \PW_0(X, \End(p_n I_\sigma))^{W_\sigma} \times [0,1] \to \PW_0(X, \End(p_n I_\sigma))^{W_\sigma},\]
	given by $\alpha_t\circ \rest(f)(\lambda) = f(\lambda/(1+t))$. Of course, when $t = 0$, the above map is the identity on $\PW_0(X, \End(p_n I_\sigma))^{W_\sigma}$. Similarly, the map 
	\[\rest \circ \alpha_\bullet: \PW_0(X/2, \End(p_n I_\sigma))^{W_\sigma} \times [0,1] \to \PW_0(X/2, \End(p_n I_\sigma))^{W_\sigma},\]
	\normalsize
	given by $(\rest \circ \alpha_t)(f)(\lambda) = f(\lambda/(1+t))$, defines a homotopy between $\rest \circ \alpha_1$ and the identity on $\PW_0(X/2, \End(p_n I_\sigma))^{W_\sigma}$. 
	
	Using the direct limit \eqref{eq: Cstar direct limit of PW}, and continuity in $K$-theory (see \cite[Theorem 3.3]{2021braddhigson}), we now see that the restriction map \eqref{eq: PW to Cstar} induces an isomorphism in $K$-theory.
\end{proof}

We now prove that the inclusion $\cS(G,F) \to C_r^*(G,F)$ induces an isomorphism in $K$-theory. We shall do so by a series of six-term exact sequence arguments and five lemma arguments. More precisely, for each $n \in \bN$, we prove that the mapping cone
\[\MC_n = \MC\left(\cJ_n \to J_n \right)\]
vanishes in $K$-theory. When $n = 0$, we have $\cJ_0 = J_0 = 0$. For $n > 0$, the short exact sequence of Fr{\'e}chet algebras
\[0 \to \MC_{n-1} \to \MC_n \to \MC(\cJ_n/\cJ_{n-1} \to J_n/J_{n-1}) \to 0,\]
leads (via \cite[Theorem 6.1]{phillipsktheory}) to the $6$-term exact sequence
\[\small
\xymatrix@C=1em{
	K_0(\MC_{n-1}) \ar[r] & K_0(\MC_n) \ar[r] &K_0(\MC(\cJ_n/\cJ_{n-1} \to J_n/J_{n-1}))\ar[d]\\
	K_1(\MC(\cJ_n/\cJ_{n-1} \to J_n/J_{n-1}))\ar[u] &\ar[l]K_1(\MC_n) & \ar[l] K_1(\MC_{n-1}).
}
\]
By Theorem \ref{thm: subquotient Kthy iso} and the above exact sequence, we see that $K_*(\MC_n) \cong K_*(\MC_{n-1})$. Inductively it follows that $K_*(\MC_n) = 0$. As $\cJ_n = \cS(G,F)$ and $J_n = C_r^*(G,F)$ for large enough $n$, this concludes the proof of Theorem \ref{thm: mainthm}.

\section{Proof of Theorem \ref{thm: Jsig minK}}\label{sec: deferredproof}
We now prove Theorem \ref{thm: Jsig minK}, which states that
\[\cJ_n/\cJ_{n-1} = (\cJ_n/\cJ_{n-1}) p_n (\cJ_n /\cJ_{n-1}).\]

We shall reduce this theorem to a ``Factoring Theorem'', which we then prove in the next section. The Factoring Theorem is the analogue for $\cS(G)$ of results of Delorme \cite{delormePW}, particularly \cite[Proposition 1]{delormePW}.

We first define the Hecke algebra. The following uses notation and results from \cite[Chapter 1]{knappvogan}.
\begin{defn}
	The Hecke algebra $R(K)$ of $K$ is the space of $K$-finite smooth functions on $K$.
\end{defn}
It is readily checked that the functions $p_\gamma$ from Definition \ref{defn: projectionontoF} are projections in $R(K)$ for each $\gamma \in \widehat{K}$.
From \cite[(1.37) and Proposition 1.39]{knappvogan},
\[R(K) \cong \bigoplus_{\gamma \in \Khat} p_\gamma R(K) \cong \bigoplus_{\gamma \in \Khat} \End(V_\gamma),\]
where $V_\gamma$ denotes a vector space representative of $\gamma \in \Khat$.

\begin{defn}\label{defn: hecke algebra}
	The Hecke algebra $R(\fg, K)$ of $G$ is the convolution algebra of $K$-finite distributions of $G$ which are supported in $K$. 
\end{defn}
By \cite[Corollary 1.71]{knappvogan}, there is an isomorphism of algebras
\begin{equation}\label{eq-heckealgasunivalg}
	R(K) \otimes_{\cU(\fk)} \cU(\fg) \xrightarrow{\cong} R(\fg, K)
\end{equation}
given by $T \otimes u \mapsto T *_K u$. Here, we identify $u \in \cU(\fg)$ with the distribution $\tilde{u} \cdot \delta_e$ supported on the identity $\{e\}$, where $\tilde{u}$ is the left-invariant differential operator corresponding to $u$.

We remark that the category of $(\fg, K)$-modules is equivalent to the category of approximately unital $R(\fg, K)$-modules (\cite[Theorem 1.117]{knappvogan}).
\begin{lem}[{See \cite[Proposition 1]{delormeHC}}]
	For $h \in R(\fg, K)$ and $\varphi, \psi \in I_\sigma$, the map $\lambda \mapsto \langle \pi_{\sigma,\lambda}(h)\varphi, \psi \rangle$ is a polynomial function on $\fa^*$.
\end{lem}

As usual, given a finite set $F \subset \widehat{K}$ we write
\[R(\fg, F) = p_F R(\fg, K) p_F.\] 

To prove Theorem \ref{thm: Jsig minK}, the important point is that the matrices $\pi_n(\cJ_n)$ can be reduced (via polynomials) to matrices on only the minimal $K$-types, stated below.
\begin{thm}[``Factoring Theorem'', {cf. \cite[Proposition 1]{delormePW}}]\label{thm: factoring thm}
	For each $n \in \bN$
	\[\pi_n(\cJ_n) = \pi_n(R(\fg,F) p_{n} \cS(G,F) p_{n} R(\fg,F)).\]
\end{thm}
The inclusion $\supseteq$ follows from the fact that $p_{\minK{\sigma_n}} \cS(G,F) p_{\minK{\sigma_n}} \subset \cJ_n$, and that $\cJ_n$ is an $R(\fg,F)$-bisubmodule of $\cS(G, F)$. The difficulty lies in the inclusion  $\subseteq$, which we prove in Section \ref{sec: delresults}. This theorem is due to Delorme \cite[Proposition 1]{delormePW} in the $C_c^\infty(G)$ case.

We recall \eqref{eq: minKfulldecomp}, which (combined with Theorem \ref{thm: PW on lowktypes}) implies that, for $\mu,\nu \in \minK{\sigma}$,
\begin{equation}
	p_{\mu} \pi_n(\cJ_n) p_\nu = \PW(\fa^*)^{W_\sigma^0}(\hat{r}_{\mu\nu}) \otimes \Hom(p_\nu I_\sigma, p_\mu I_\sigma).
\end{equation}
Let us briefly recall the notation in the above equation. The group $W_\sigma$ decomposes as a semidirect product of subgroups
\[W_\sigma = R_\sigma W_\sigma^0,\]
where $R_\sigma$ is a product of copies of $\bZ/2$. Also, there is a simply transitive action of $\widehat{R}_\sigma$ on $A(\sigma)$, and we write $\hat{r}_{\mu\nu}$ to denote the unique element of $\widehat{R}_\sigma$ such that $\hat{r}_{\mu\nu} \cdot \nu = \mu$. Accordingly, we define $\PW(\fa^*)^{W_\sigma^0}(\hat{r}_{\mu\nu})$ to consist of all $W^0_\sigma$-invariant $f$ such that
\[f(w\lambda) = \hat{r}_{\mu\nu}(w)f(\lambda)\]
for any $w \in W_\sigma$ (we have extended $\hat{r} \in \widehat{R}_\sigma$ to $W_\sigma$ via the decomposition $W_\sigma = R_\sigma W_\sigma^0$.)

We now proceed toward the proof of Theorem \ref{thm: Jsig minK}, which states that
\[(\cJ_n/\cJ_{n-1})p_n(\cJ_n/\cJ_{n-1}) = \cJ_n/\cJ_{n-1}.\]

\begin{lem}\label{lem: dixmiermalliavin}
	We have
	\[ \PW(\fa^*) \cdot  \PW(\fa^*) =  \PW(\fa^*).\]
\end{lem}
\begin{proof}
	Let $A$ act on $\PW(\fa^*)$ by $(a \cdot f)(\lambda) = e^{\lambda(\log a)} f(\lambda)$. Integrating this representation, the Casselman algebra, $\cS(A)$, of $A$ acts on $\PW(\fa^*)$ as multiplication by the (Euclidean) Fourier transform, and the lemma is implied by the statement that $\cS(A) \cdot \PW(\fa^*)  = \PW(\fa^*)$. This is now a consequence of \cite[Remark 2.19]{2014bernsteinkrotz}.
\end{proof}

\begin{lem}\label{lem: raisspecific}
	We have
	\[ \PW(\fa^*) = \bC[\fa^*]  \PW(\fa^*)^{W(\fa)}.\]
\end{lem}
This is a consequence of Theorem \ref{thm: Rais Theorem}.

\begin{lem}\label{lem: dixmiermal averaged}
	For each cuspidal pair $(P,\sigma)$ and each $\mu,\nu \in \minK{\sigma}$,
	\begin{equation}\label{eq: dixmier-malliavin}
		\PW(\fa^*)^{W_\sigma^0}(\hat{r}_{\mu\nu}) \cdot  \PW(\fa^*)^{W_\sigma} = \PW(\fa^*)^{W_\sigma^0}(\hat{r}_{\mu\nu}).
	\end{equation}
\end{lem}
\begin{proof}
	Because $1 \in \bC[\fa^*]$, Lemmas \ref{lem: dixmiermalliavin} and \ref{lem: raisspecific} imply
	{
		\[ \PW(\fa^*)  \PW(\fa^*)^{W(\fa)} =  \PW(\fa^*) \bC[\fa^*]  \PW(\fa^*)^{W(\fa)} =  \PW(\fa^*)  \PW(\fa^*) =  \PW(\fa^*),\]}and we obtain \eqref{eq: dixmier-malliavin} by averaging by $W_\sigma^0$ and projecting onto the $\widehat{R}_\sigma$-isotypical component $\hat{r}_{\mu\nu}$ (note that this projection commutes with multiplication by $ \PW(\fa^*)^{W(\fa)}$).
\end{proof}

\begin{proof}[Proof of Theorem \ref{thm: Jsig minK}]
	We will prove that
	\[\overline{\pi}_n(\cJ_n/\cJ_{n-1}) = \overline{\pi}_n(\cJ_n/\cJ_{n-1})p_{n} \overline{\pi}_n(\cJ_n/\cJ_{n-1}),\]
	which implies the theorem because $\overline{\pi}_n$ is an injective algebra homomorphism. Note that this is equivalent to the statement
	\[\pi_n(\cJ_n) = \pi_n(\cJ_n) p_n \pi_n(\cJ_n).\]
	
	Set $\sigma = \sigma_n$. It suffices to prove that 
	\begin{equation}\label{eq: morita by insert hecke}
		\pi_n(\cS(G,\minK{\sigma})) = \pi_n(\cS(G,\minK{\sigma})R(\fg,\minK{\sigma})\cS(G,\minK{\sigma})),
	\end{equation}
	because then, by Theorem \ref{thm: factoring thm} and the fact that $p_n \in R(\fg,\minK{\sigma})$,
	\begin{align*}
		\pi_n(\cJ_n) p_{n} \pi_n(\cJ_n) &=
		\pi_n(R(\fg,F)\cS(G,\minK{\sigma})R(\fg,\minK{\sigma})\cS(G,\minK{\sigma})R(\fg,F))\\
		&= \pi_n(R(\fg,F)\cS(G,\minK{\sigma})R(\fg,F)) = \pi_n(\cJ_n).
	\end{align*}
	We make use of the explicit formula for $\mu,\nu \in \minK{\sigma}$,
	\[\pi_n(p_\mu \cS(G,\minK{\sigma})p_\nu) =  \PW(\fa^*)^{W_{\sigma}^0}(\hat{r}_{\mu\nu}) \otimes \Hom(I_{\sigma}(\nu), I_{\sigma}(\mu)) \]
	which follows from Theorem \ref{thm: PW on lowktypes} and \eqref{eq: minKfulldecomp}.
	Applying \eqref{eq: dixmier-malliavin} on matrices, we obtain
	\[\pi_n(p_\mu \cS(G,\minK{\sigma})p_\nu) = \pi_n(p_\mu \cS(G,\minK{\sigma})p_\nu\cS(G,\minK{\sigma})p_\nu).\]
	As $p_\nu$ acts as the identity on $p_\nu R(\fg,\minK{\sigma})p_\nu$, we have
	\[\pi_n(p_\mu \cS(G,\minK{\sigma})p_\nu) = \pi_n(p_\mu \cS(G,\minK{\sigma})p_\nu R(\fg, \minK{\sigma})p_\nu\cS(G,\minK{\sigma})p_\nu).\]
	This gives the inclusion $\subseteq$ of \eqref{eq: morita by insert hecke}, and the other inclusion $\supseteq$ follows because $\pi_\sigma(\cS(G,\minK{\sigma_n}))$ is closed under the left and right action of $R(\fg,\minK{\sigma_n})$.
\end{proof}

\section{Proofs of Delorme's Factoring Theorem and Divisibility Theorem for $\cS(G, F)$}\label{sec: delresults}

We now turn to the proof of Theorem \ref{thm: factoring thm}, which states that
\[\pi_n(\cJ_n) = \pi_n\left(R(\fg, F)\cS(G,\minK{\sigma})R(\fg,F)\right).\]

As it will be important to consider arbitrary parabolic subgroups with some fixed Levi subgroup, we recall that we chose representatives $(P_n, \sigma_n)$ for each $G$-conjugacy class. We will fix $n$ such that $A(\sigma_n) \subset F$, and write $P_n = MAN$, $\sigma = \sigma_n$.

\begin{defn}\label{defn: adjacent parabolics}
	Two parabolic subgroups $P$ and $Q$ with Levi subgroup $MA$ are \textit{adjacent} if $\Delta^+_P \cap -\Delta_Q^+$ has a unique reduced root. If $\alpha$ is this root, and if $\lambda \in \fa^*$, then $\lambda_\alpha$ will denote the projection of $\lambda$ onto $\bC\alpha \subset \fa^*$ with respect to the Killing form.
\end{defn}
When $P$ and $Q$ are adjacent, the operator $A(Q, P, \sigma, \lambda)$ (and its normalized version) depends only on $\lambda_\alpha$. Indeed, $\theta({N_P}) \cap N_Q$ can be regarded as the ``$\theta(N)$'' of $G_\alpha = Z_G(\ker \alpha)$, which has split rank $1$ (see \cite[VII.6]{knappbeyondintro}). Then $a_P(\nbar) \in G_\alpha$, so that $a_P(\nbar)^\lambda = a_P(\nbar)^{\lambda_\alpha}$. Alternatively, this can be deduced by an induction in stages formula (see \cite[(1.4)]{delormePW}).

\begin{defn}\label{defn: Esigma}
	We define $\PW_{\text{divis}}(\fa^*, \End(p_F I_\sigma))$ to be the set of functions $f \in \PW(\fa^*, \End(p_F I_\sigma))$ with the following divisibility properties:
	\begin{enumerate}
		\item For each parabolic subgroup $P$ with Levi subgroup $MA$, there exists a (unique) function $f^P \in \PW(\fa^*, \End(p_F I_\sigma))$ such that
		\[\cA(P,P_n,\lambda)f(\lambda) = f^P(\lambda)\cA(P,P_n,\lambda).\]
		\item For each $w \in W_\sigma$, the map $f^P$ satisfies
		\[\cA(P, w, \lambda)f^P(\lambda) = f^P(w\lambda)\cA(P,w,\lambda).\]
		\item  Let $P$ and $Q$ be adjacent parabolic subgroups with Levi subgroup $MA$. The map $\lambda \mapsto f^P(\lambda)\cA(Q,P,\lambda)^{-1}$, initially meromorphic on $\fa^*$, extends to a holomorphic function on a neighborhood of $\overline{\fa^*_+}$.
	\end{enumerate}
\end{defn}
Properties $1$ and $2$ are based on the fact that, for $\phi \in \cS(G,F)$, we have the intertwining relations {\small
	\[\small \cA(P,Q,\lambda)\pi^Q_{\sigma,\lambda}(\phi) = \pi^P_{\sigma,\lambda}(\phi) \cA(P,Q,\lambda), \quad \cA(P,w,\lambda)\pi^P_{\sigma,\lambda}(\phi) =\pi^P_{\sigma,w\lambda}(\phi)\cA(P,w,\lambda).\]
}An important step in the proof of Theorem \ref{thm: factoring thm} is to show that elements of $\pi_n(\cJ_n)$ satisfy Property 3.

In the following, we set $\sigma = \sigma_n$, and we assume $F$ contains $A(\sigma)$.
\begin{thm}[Factoring Theorem]\label{thm: true factoring thm}
	\[\PW_{\emph{\text{divis}}}(\fa^*, \End(p_FI_\sigma))\subseteq \pi_\sigma(R(\fg,F)\cS(G,\minK{\sigma})R(\fg,F)).\]
\end{thm}

\begin{thm}[Divisibility Theorem]\label{thm: divisibility thm}
	\[\pi_\sigma(\cJ_n) \subseteq \PW_{\emph{\text{divis}}}(\fa^*, \End(p_FI_\sigma)).\]
\end{thm}

The following lemma implies that the above inclusions are equalities.
\begin{lem}
	We have
	\[\pi_\sigma(R(\fg,F)\cS(G,\minK{\sigma})R(\fg,F)) \subseteq \pi_\sigma(\cJ_n).\]
\end{lem}
\begin{proof}
	This is a consequence of the fact that $\pi_{\tau}(p_{\minK{\sigma}}) = 0$ when $[\tau] > [\sigma]$. 
\end{proof}

The two theorems and the lemma imply that
\[\pi_\sigma(\cJ_n) = \PW_{\text{divis}}(\fa^*, \End(p_FI_\sigma)) =\pi_\sigma(R(\fg,F)\cS(G,\minK{\sigma})R(\fg,F)),\]
which implies Theorem \ref{thm: factoring thm}.

In the $C_c^\infty(G)$ case, the Divisibility and Factoring Theorems are essentially \cite[Proposition 1, and (3.8)]{delormePW}, and our proofs are almost identical. In fact, the only real difference is the use of polynomial division on $ \PW(\fa^*)$ (see Lemma \ref{lem: polydivision}), whose proof is practically the same as in the $C_c^\infty(G)$ case (due to Clozel and Delorme \cite{delormeclozel2}).

It is illuminating to view these two theorems in the context of the spherical principal series, where the minimal $K$-type is the trivial $K$-type, which is the case $n = 1$. We will consider the example $G = \SL(2,\bR)$.

\begin{ex}
	Let $G = \SL(2, \bR)$ and $K = \SO(2)$. The two relevant parabolic subgroups are the minimal parabolic subgroup $P$ consisting of upper triangle matrices in $G$, and $G$ itself. We set 
	\[M = \{\pm \Id\}, \quad A =\{\diag(e^t, e^{-t}): t \in \bR\}, \quad  N = \left\{\begin{bmatrix*}1 &s\\0&1\end{bmatrix*}: s\in \bR\right\},\]
	so that $P = MAN$ is the Langlands decomposition of $P$.
	
	We identify $\Khat$ with $\bZ$. We then set $F = \{-2, 0, 2\} \subset \widehat{K}$. There are three relevant cuspidal pairs whose corresponding principal series contains $K$-types in $F$.
	\begin{itemize}
		\item $(P, \sigma)$, where $\sigma = \sigma_1$ be the trivial representation of $M$. The corresponding principal series $I_\sigma$ is known as the spherical principal series.
		\item $(G, \sigma_2)$, where $\sigma_2 = (\pi_{D_{2,+}}, D_{2, +})$ is the discrete series whose minimal $K$-type is $2 \in \widehat{K}$. In this case, $\pi_{\sigma_2} = \pi_{D_{2,+}}$.
		\item $(G, \sigma_3)$, where $\sigma_3 = (\pi_{D_{2,-}}, D_{2, -})$ is the discrete series whose minimal $K$-type is $-2 \in \widehat{K}$. In this case, $\pi_{\sigma_3} = \pi_{D_{2,-}}$.
	\end{itemize}
	We also will use the fact that we have an embedding of representations,
	\[(\pi_{D_{2,+}} \oplus \pi_{D_{2,-}}, D_{2,+} \oplus D_{2, -}) \to (\pi^P_{\sigma, 1}, I_\sigma).\]
	
	Given $k_\theta = e^{i\theta}\in K$, we define $e_n(k_\theta) = e^{-in\theta}$. As a $K$-representation, $I_\sigma$ consists of even $K$-types with multiplicity $1$. That is, $I_\sigma$ is spanned by $e_n$ for even $n$. Accordingly, $p_F I_\sigma$ has ordered basis $e_2, e_0, e_{-2}$, while $p_F D_{2,+}$ has basis $e_2$ and $p_F D_{2,-}$ has basis $e_{-2}$ (which we will identify inside of $I_\sigma$). Using this ordered basis, we identify $p_F I_\sigma$ with $\bC^3$.
	
	We will parametrize $\fa^*$ as follows. Set $H_s = \diag(1, -1) \in \fa_0$. We then identify $\fa^*$ with $\bC$ via $\lambda \mapsto \lambda(H_s)$. Under our identifications, the positive root in $\Delta(\fg_0, \fa_0)$ is the usual $\alpha = 2$.
	
	With the above identifications, elements of $\pi_\sigma(\cS(G,F))$ identify with maps $(z \mapsto f(z)) \in \PW(\bC, M_3(\bC))$. We will now write down this image and verify Theorems \ref{thm: true factoring thm} and \ref{thm: divisibility thm} in this context.
	
	It is simpler to compute $\pi_\sigma(R(\fg, F))$ by completely algebraic means. From the identification $R(\fg, F) = \cU(\fg) \otimes_{\cU(\fk)} R(K)$, and from the formulas given in \cite[Proposition 31, p. 136]{varadarajanbook}, we can explicitly compute that 
	\[\pi_\sigma(R(\fg, F)) = \left\{\begin{bmatrix*}p_{2,2}(z^2) & (z+1)p_{2,0}(z^2)& (z^2-1)p_{2,-2}(z^2)\\
		(z-1)p_{0,2}(z^2)&p_{0,0}(z^2)&(z-1)p_{0,-2}(z^2)\\
		(z^2-1)p_{-2,2}(z^2) & (z+1)p_{-2,0}(z^2)& p_{-2,-2}(z^2)
	\end{bmatrix*}\right\},\]
	where  $p_{i,j} \in \bC[z]$ for $i,j \in F$. The above matrices are with respect to the ordered basis $e_2, e_0, e_{-2}$ of $p_F I_\sigma$.
	
	It can also be shown that $\pi_\sigma(\cS(G,F))$ has the same form as above, with the $p_{i,j}(z^2)$ replaced by even functions $f_{i,j}(z)$ in $\PW(\bC)$. We see that the entries have particular algebraic relations. For example, the top right entry is an even function with guaranteed zeros at $z = \pm 1$.
		
	In our case of $\SL(2,\bR)$, it is easy to see where the algebraic relations come from. First of all, the Weyl group $W_\sigma = W(\fg_0, \fa_0)$ consists of two elements, where the nontrivial element $w$ acts on $\fa^* \cong \bC$ via $z \mapsto -z$. The corresponding action of $\cA(P, w, \sigma, z)$ (where we have normalized $A(P,w, \sigma,z)$ with respect to the trivial $K$-type, and restricted the operator to $p_F I_\sigma$) is identified as the matrix
	\[\cA(P,w,\sigma, z) =\begin{bmatrix*}\dfrac{z-1}{z+1}&0&0\\0&1&0\\0&0&-\dfrac{z-1}{z+1}\end{bmatrix*}.\]
	on $p_F I_\sigma$ with respect to the basis $e_2, e_0, e_{-2}$. The algebraic conditions follow from the intertwining relation
	\[f(z) = \cA(P,w,-z)f(-z)\cA(P,w,z).\]
	In other words, in this case it holds that $\pi_\sigma(\cS(G,F)) = \PW(\fa^*, \End(p_F I_\sigma))^{W_\sigma}$ (we caution that this is does not necessarily hold for more general groups).
	
	Let us now describe the ideal $\pi_\sigma(\cJ_1)$. The ideal $\cJ_1$, which is $\ker(\pi_{\sigma_2}) \cap \ker(\pi_{\sigma_3})$, is given (under Fourier transform) by
	\[\pi_\sigma(\cJ_1) = \left\{\begin{bmatrix*}(z^2-1)f_{2,2}(z) & (z+1)f_{2,0}(z)& (z^2-1)f_{2,-2}(z)\\
		(z-1)f_{0,2}(z)&f_{0,0}(z)&(z-1)f_{0,-2}(z)\\
		(z^2-1)f_{-2,2}(z) & (z+1)f_{-2,0}(z)& (z^2-1)f_{-2,-2}(z)
	\end{bmatrix*}\right\},\]
	where the functions $f_{i,j}$ are even functions in $\PW(\bC)$. Indeed, when $\phi \in \cJ_1$, we see that (identifying $D_{2, \pm}$ inside of $I_{\sigma,1}$)
	\[f_{2,2}(1) = p_2\pi^P_{\sigma,1}(\phi)p_2 = \pi_{D_{2,+}}(\phi)p_2 = 0,\]
	and similarly we have $f_{-2,-2}(1) = 0$. Because both $f_{2,2}$ and $f_{-2,-2}$ are even, we can therefore factor out $(z^2-1)$ from these functions.
	
	Now, let $\overline{P} = \theta(P)$ be the opposite parabolic subgroup (consisting of lower-triangular matrices in $G$). Again, we write $\cA(\overline{P}, P, \lambda)$ for the normalization of $A(\overline{P}, P, \sigma, \lambda)$ with respect to the trivial $K$-type, restricted to $p_F I_\sigma$. The Divisibility Theorem states that $\pi_\sigma(\phi)(z)\cA(\overline{P}, P, z)^{-1}$ is holomorphic for $\Re z \geq 0$. This is easily verified by the explicit description of $\pi_\sigma(\cJ_\sigma)$ and the fact that
	\[\cA(\overline{P}, P, z) = \begin{bmatrix*}\dfrac{z-1}{z+1}&0&0\\0&1&0\\0&0&-\dfrac{z-1}{z+1}\end{bmatrix*}.\]
	
	The Factoring Theorem is the statement that the whole of $\pi_\sigma(\cJ_1)$ can be obtained by applying elements of the Hecke algebra on the left and right to elements of the form
	\[\begin{bmatrix*}0&0&0\\0&f_{0,0}(z)&0\\0&0&0\end{bmatrix*},\]
	and taking the span. The above is easily verified by our explicit calculation of $\pi_\sigma(R(\fg, F))$ and $\pi_\sigma(\cJ_1)$. Of course, from our calculation of $\pi_\sigma(\cJ_1)$ we can see directly that  $\pi_\sigma(\cJ_1) = \pi_\sigma(\cJ_1)p_{\{0\}} \pi_\sigma(\cJ_1)$ (which is Theorem \ref{thm: Jsig minK} in our case).
\end{ex}

\newcommand{\ldot}{{}_\cdot}
\newcommand{\ldott}{{}_\cdot {}_\cdot}
\subsection{Proof of the Divisibility Theorem}
The proof of the Divisibility Theorem relies on the theory of derivatives for holomorphic families of representations, developed by Delorme and Souaifi. See \cite{delormesouaifi}, and see \cite{vdbSouaifi} for a systematic treatment of this theory. The Divisibility Theorem is essentially a consequence of the following theorem of Delorme and Souaifi.
\begin{thm}[See {\cite[Theorem 3 (ii)]{delormesouaifi}}]\label{thm: DelSou}
	Let $X$ be an admissible $(\fg, K)$-module whose $K$-types have length larger than $R$. Then $X$ is a subquotient of a direct sum of (successive derivatives of) principal series representations, each of which contain only $K$-types of length larger than $R$.
\end{thm}
We define what is meant by differentiation below.

\begin{defn}[See \cite{delormePW} and \cite{vdbSouaifi}]
	Let $\Omega$ be an open subset of $\bC$. Given an operator-valued holomorphic function $A: \Omega \to \End(V)$, we define
	\[\Delta_z A: \Omega \to \End(V \oplus V)\]
	by the block matrix formula
	\begin{align*}
		\Delta_z A(z) &= \begin{bmatrix*}A(z) & A'(z) \\ 0& A(z)\end{bmatrix*}.
	\end{align*}
	
\end{defn}
\begin{lem}\label{lem: derivproductrule}
	Let $A,B: \Omega \to \End(V)$ be holomorphic. Then 
	\[\Delta_z(A(z)B(z)) = (\Delta_z A(z))(\Delta_z B(z)).\]
\end{lem}
\begin{proof}
	This follows from the Leibniz rule.
\end{proof}
We must also account for several variables. The following defines a partial derivative for operator-valued holomorphic functions.
\begin{defn}\label{defn: oppartialderivative}
	Let $\Omega \subset \bC^n$ be an open subset, and let $A: \Omega \to \End(V)$ be holomorphic. Given $w \in \bC^n$, we define
	\[\Delta_w A: \Omega \subset \bC \to \End(V \oplus V) \]
	by
	\[\Delta_w A(z) = \Delta_{z'} A_{z,w}(0),\]
	where $A_{z,w}: \Omega' \subset \bC \to \End(V)$ is defined by $A_{z,w}(z') = A(z+z'w)$, and where $\Delta_{z'} A_{z,w}$ is defined as in the previous definition.
	
	We also use the same definition for holomorphic functions on complex vector spaces. Given a coordinate system $(z_1, \ldots, z_n)$ of this vector space, we may use the notation $\Delta_{z_i}$ instead of $\Delta_{e_i}$, where $e_i$ is the corresponding basis.
\end{defn}

Using this definition of derivative, we may now differentiate the principal series (see \cite{vdbSouaifi} and \cite{delormesouaifi}). Higher derivatives are obtained by successively applying $\Delta_w$.

\begin{thm}\label{thm: delsouconsequence}
	Elements of $\cJ_n$ act by $0$ on any admissible $(\fg,K)$-module $V$ whose $K$-types have length larger than $\|\sigma_n\|$.
\end{thm}
\begin{proof}
	By definition of $\cJ_n$ and our total ordering on $G$-conjugacy classes $[\sigma]$, such elements $\phi \in \cJ_n$ act by zero on principal series representations whose $K$-types have length larger than $\|\sigma_n\|$.	By our definition of $\Delta^N_z$, it follows that $\phi$ acts by $0$ on the corresponding (successive) derivative representations of these principal series.
	
	According to Theorem \ref{thm: DelSou}, $V$ is a subquotient of a direct sum of derivatives of principal series representations, each of whose $K$-types have length larger than $\|\sigma_n\|$. It follows that $\phi$ acts by $0$ on $V$.
\end{proof}
Turning to the Divisibility Theorem, we need one more general lemma.
\begin{lem}\label{lem: generaldivisibility}
	Let $V$ be a finite-dimensional vector space. Let $f: \bC \to \End(V)$ be holomorphic, and let $A: \Omega \to \End(V)$ be a rational function defined and holomorphic in a neighborhood $\Omega \subset \bC$ of $0$. If $\Delta_{z}^N f(0)$ vanishes on $\ker (\Delta_{z}^N A(0))$ for each $N$, then $f(z)A(z)^{-1}$ is holomorphic in a neighborhood of $0$.
\end{lem}
\begin{proof}
	Let $v \in V$ and suppose $z^{N+1}$ divides $\cA(z)v$. Then $\Delta^N_z A(0)v = 0$ and therefore $\Delta_{z}^Nf(0)v = 0$. Unpacking the definition of $\Delta_{z}^N$, this means that $(z \mapsto f(z)v)^{(k)}(0) = 0$ whenever $k \leq N$, and therefore $z^{N+1}$ divides $f(z) v$. It follows that we can define $f(z) A(z)^{-1}v$ in a neighborhood of $0$, which will be holomorphic in $z$.
\end{proof}

We must also deal with several variables. The following lemma generalizes the previous lemma to this case.
\begin{lem}\label{lem: supergeneraldivisibility}
	Let $V$ be a finite-dimensional vector space. Let $f: \bC^n \to \End(V)$ be holomorphic, and let $A: \Omega \to \End(V)$ be a rational function defined and holomorphic in a neighborhood $\Omega \subset \bC$ of $0$. Embed $\bC$ into $\bC^n$ via the first coordinate, for which we write $\lambda = (\lambda_1, \lambda') \in \bC^n$, where $\lambda_1 \in \bC$ and $\lambda' \in \bC^{n-1}$.
	
	If $\Delta_{\lambda_1}^N f(0, \lambda')$ vanishes on $\ker (\Delta_{\lambda_1}^N A(0))$ for each $N$ and for some $\lambda' \in \bC^{n-1}$, then $
	\lambda \mapsto f(\lambda)A(\lambda_1)^{-1}$ is holomorphic in a neighborhood of $(0,\lambda') \in \bC^n$.
\end{lem}
This follows from the previous lemma. The point is that, if $\lambda_1^{N}$ divides $A(\lambda_1)v$, then $\lambda_1^N$ divides $f(\lambda)v$ (where $v \in V$).

\begin{proof}[Proof of the Divisibility Theorem] See \cite[(3.8)]{delormePW}. Fix $\phi \in \cJ_n$ and set $f = \pi_\sigma(\phi)$. Let $P$ and $Q$ be adjacent parabolic subgroups containing $MA$. Note that (defining $f^P$ as in Definition \ref{defn: Esigma})
	\[f^P(\lambda) = \pi^P_{\sigma,\lambda}(\phi).\]
	
	We treat the unique reduced $\alpha \in \Delta^+_P \cap -\Delta^+_Q$ as the first coordinate of $\fa^*$, where we extend $\{\alpha\}$ to some basis. We note that $\cA(Q,P,\lambda)$ depends only on $\lambda_\alpha$ (see Definition \ref{defn: adjacent parabolics}).
	We must show that $f^P(\lambda)\cA(Q,P,\lambda_\alpha)^{-1}$ is holomorphic in a neighborhood of each $\lambda \in \overline{\fa^*_{P,+}}$.
	
	According to Lemma \ref{lem: supergeneraldivisibility}, it suffices to prove that $\Delta^N_{\lambda_\alpha} f^P(\lambda)$ vanishes on the space $\ker(p_F \Delta^N_{\lambda_\alpha}\cA(Q,P,\lambda_\alpha)) = p_F \ker(\Delta^N_{\lambda_\alpha}\cA(Q,P,\lambda_\alpha))$ whenever $\Re \lambda_\alpha \geq 0$. Now, the $(\fg,K)$-module \[\ker(\Delta^N_{\lambda_\alpha}\cA(Q,P,\lambda_\alpha))\] is a submodule of $(\Delta^N_{\lambda_\alpha} \pi^P_{\sigma_n, \lambda}, I_{\sigma_n}^{\oplus N+1})$. Moreover, because $\cA(Q, P, \lambda_\alpha)$ is constant and nonzero on the minimal $K$-types of $I_{\sigma_n}$, the operators $\Delta^N_{\lambda_\alpha} \cA(Q,P,\lambda_\alpha)$ are also constant and nonzero on the minimal $K$-types of $(I^P_{\sigma_n})^{\oplus N + 1}$. Therefore, $\ker(\Delta^N_{\lambda_\alpha}\cA(Q,P,\lambda_\alpha))$ does not contain any $K$-type in $A(\sigma_n)$, and therefore its $K$-types have length larger than $\|\sigma_n\|$. By Theorem \ref{thm: delsouconsequence}, $\phi$ acts by $0$ on $\ker(\Delta^N_{\lambda_\alpha}\cA(Q,P,\lambda_\alpha))$. But $\phi$ acts precisely by $\Delta^N_{\lambda_\alpha} f^P(\lambda)$, so we are done. 
\end{proof}


\newcommand{\tmu}{\tilde{\mu}}
\newcommand{\tv}{\tilde{v}}
\newcommand{\tphi}{\tilde{\phi}}

\newcommand{\tPhi}{\widetilde{\Phi}}

\renewcommand{\tPsi}{\widetilde{\Psi}}
\subsection{Proof of the Factoring Theorem}
Our aim is to prove that
\begin{equation}\label{eq: facpropeq}
	\PW_{\text{divis}}(\fa^*, \End(p_F I_\sigma)) \subseteq \sum_{\mu,\nu \in \minK{\sigma}} \pi_\sigma(R(\fg,F) p_{\mu}\cS(G,\minK{\sigma})p_\nu R(\fg,F))
\end{equation}
(recall Definition \ref{defn: Esigma}).

Let $u \in \PW_{\text{divis}}(\fa^*, \End(p_F I_\sigma))$. We recall that we have fixed a representative $(P_n = MAN, \sigma = \sigma_n)$ of $[\sigma]$, and that for any parabolic subgroup $P$ with Levi subgroup $MA$, the function
\[u^P(\lambda) = \cA(P, P_n, \lambda)u(\lambda)\cA(P_n, P, \lambda)\]
defines an element of $\PW(\fa^*, \End(p_FI_\sigma))$.

{Most details of the proof of the Factoring Theorem can be found in \cite[Section 2]{delormePW}, which treats $C_c^\infty(G)$ instead of $\cS(G)$. As a result, we will state only what is necessary to cite Delorme's results. However, for the benefit of the reader, we briefly outline the details found in \cite[Section 2]{delormePW}.
	
	We wish to find a decomposition
	\[u = \sum \phi_i M_{ij} \tilde{\phi}_j,\] where $\phi_i \in \pi_\sigma(R(\fg,F)p_{\mu_i})$, $\tilde{\phi_j} \in \pi_\sigma(p_{\tmu_j}R(\fg,F)))$, $M_{ij} \in \pi_\sigma(p_{\mu_i}\cS(\fg,\minK{\sigma})p_{\tmu_j})$, and $\mu_i, \tmu_j \in A(\sigma)$. If we fix any choice of such $\phi_i, \tilde{\phi}_j$, then this becomes a linear algebra equation with respect to (the matrix components of) $M_{ij}$, over the field of meromorphic functions on $\fa^*$, where we are treating $M_{ij}$ as matrices via 
	\[\pi_\sigma(p_{\mu_i}\cS(G,\minK{\sigma})p_{\tmu_j})= \PW(\fa^*)^{W_\sigma^0}(\hat{r}_{\mu_i\tmu_j}) \otimes \Hom(p_{\mu_i}I_\sigma, p_{\tmu_j}I_\sigma)\]
	as a consequence of Theorem \ref{thm: PW on lowktypes} and \eqref{eq: minKfulldecomp}.
	
	By use of Cramer's rule, we obtain $M_{ij}$ as meromorphic functions of the form $p/q$, where $p \in \PW(\fa^*)$, and $q$ is a determinant term which depends on $\phi_i, \tilde{\phi}_j$. By understanding these determinants, we find that $M_{ij}$ is holomorphic (hence in $\PW(\fa^*)$) for certain $u$. We then decompose $u$ into pieces where the above is possible. We will have obtained
	$u = \sum \phi_i N_{ij} \tilde{\phi_j}$
	where $N_{ij} \in \PW(\fa^* )\otimes \Hom(I_\sigma(\mu_j), I_\sigma(\mu_i))$. Finally, to obtain the $W_\sigma$-invariance, we must average using $\cA(P_n,w,\lambda)$ for $w \in W_\sigma$, which will provide us with the desired decomposition of $u$.}

\begin{defn}\label{defn: factoring data}
	Let $l = \dim p_F I_\sigma$.
	
	By \textit{left factoring data}, we refer to a list of tuples $(\mu_i, v_i, \phi_i)_{i=1}^{l}$ such that $\mu_i \in \minK{\sigma}$, $v_i$ is a unit vector in $p_{\mu_i} I_\sigma$, and $\phi_i \in \pi_\sigma(R(\fg,F)p_{\mu_i})$. Corresponding to this data is the vector space $V = \bigoplus \bC v_i$.
	
	By \textit{right factoring data}, we refer to a list of tuples $(\tmu_j, \tv_j,\tphi_j)_{j=1}^{l}$ with $\tmu_j \in \minK{\sigma}$, $\tv_j$ a unit vector in $p_{\tmu_j}I_\sigma$, and $\tphi_j \in \pi_\sigma(p_{\tmu_j}R(\fg,F))$. Correspondingly, we define $\widetilde{V} = \bigoplus \bC \tv_j$.
	
	Given a set of left factoring data $(\mu_i, v_i, \phi_i)$ and parabolic subgroup $P$ with Levi subgroup $MA$, define the operator
	\[\Phi^P = \Phi^P(\lambda): V \to p_FI_{\sigma} \]
	by $\Phi^P(\lambda)(v_i) = \pi_{\sigma,\lambda}^P(\phi_i)v_i$. Given right factoring data $(\tmu_j, \tv_j, \tphi_j)$, we define
	\[\tPhi^P = \tPhi^P(\lambda) : p_FI_\sigma \to \widetilde{V} \]
	by $\tPhi^P(\lambda)(\psi)= \sum_j \langle \pi^P_{\sigma,\lambda}(\tphi_j) \psi, \tv_j\rangle_{L^2(K)} \tv_j$.
	
	Finally, given left and right factoring data, we define the ``elementary matrices''
	\[E_{ij}: \widetilde{V} \to {V}\]
	by $E_{ij} (\tv_j) = v_i$ and $E_{ij}(\tv_k) = 0$ for $k \neq j$.
\end{defn}
\begin{lem}[{See \cite[Lemmas 2, 3, and 9]{delormePW}}]\label{lem: specializeddelorme}
	Let $P$ be a parabolic subgroup with Levi subgroup $MA$.
	\begin{enumerate}
		\item There exists a polynomial $b^P \in \bC[\fa^*]$, nonzero on $-\overline{\fa^*_{P,+}}$, such that
		\[\det(p_F \cA(\overline{P}, P, \lambda)) = c\frac{b^P(\lambda)}{\bbar^P(-\lambda)},\]
		where $\overline{b}^P(\lambda) = \overline{b^P(\overline{\lambda})}$, and $c \in \bC$ is a constant with modulus $1$
		\item There exists polynomials $\Psi^P, \tPsi^P \in \bC[\fa^*]$ such that
		\begin{align*}
			\det \Phi^P(\lambda) = \Psi^P(\lambda) \bbar^P(-\lambda), \quad \det \tPhi^P(\lambda) = \tPsi^P(\lambda) b^P(\lambda).
		\end{align*}
		Moreover, if $Q$ is another parabolic subgroup, then $\Psi^P$ and $\Psi^Q$ are related by a nonzero constant (and similarly for $\tPsi^P$ and $\tPsi^Q$).
		\item The span of $\Psi^P$ across all left factoring data equals $\bC[\fa^*]^{W_\sigma^0}$. The span of $\tilde{\Psi}^P$ across all right factoring data equals $\bC[\fa^*]^{W_\sigma^0}$.
	\end{enumerate}
\end{lem}
The above lemma is a specialization of \cite[Lemmas 2, 3, and 9]{delormePW}, which we have provided for context in order to state the next lemma, as well as to prove the Factoring theorem.

\begin{lem}\label{lem: main factoring lemma}
	Let $u \in \PW_{\emph{\text{divis}}}(\fa^*, \End(p_F I_\sigma))$. Fix factoring data $(\mu_i, v_i, \phi_i)$ and $(\tmu_j, \tv_j, \tphi_j)$ with corresponding $\Phi^P, \tPhi^P$. There exists functions
	\[M^P_{ij} = M^P_{ij}(\lambda) \in  \PW(\fa^*)\]
	such that, with $M^P = \sum_{i,j} M^P_{ij} E_{ij}$, 
	\begin{equation}\label{eq: factoring out u into M}
		\Psi^P \tPsi^P u^P = \Phi^P M^P \tPhi^P.
	\end{equation}
\end{lem}
This is the analogue of \cite[Lemma 6]{delormePW}. If we ignore the rapidly decreasing condition (i.e. that $M_{ij}^P \in \PW(\fa^*)$), then this states that $M_{ij}^P$ is a holomorphic function, which is \cite[Lemma 7]{delormePW}. We do not prove that $M^P_{ij} \in \PW(\fa^*)^{W_\sigma^0}(\hat{r}_{\mu_i \mu_j})$ (stated analogously in \cite[Lemma 6]{delormePW}), because this is not true in general. Instead, the issue of $W_\sigma$-invariance will be dealt with in the proof of the Factoring Theorem.

\begin{proof}
	By \cite[Lemma 7, p. 1013]{delormePW}, if $N^P$ denotes the solution to
	\[u^P = \Phi^P N^P \tPhi^P,\]
	then $\Psi^P \tPsi^P N^P$ is a holomorphic function (this only uses the divisibility properties listed in Definition \ref{defn: Esigma}). Now, according to Cramer's rule, there exists functions $p_{ij}^P \in \PW(\fa^*)$ such that, writing $N^P = \sum E_{ij} N_{ij}^P$, then
	\[N_{ij}^P(\lambda) = \frac{p_{ij}^P(\lambda)}{\det \Phi^P(\lambda) \det \tPhi^P(\lambda)}.\]
	Then the fact that $\Psi^P \tPsi^P N_{ij}^P$ is holomorphic implies that $\det \Phi^P(\lambda) \det \tPhi^P(\lambda)$ divides $\Psi^P \tPsi^P p_{ij}^P$. According to the Lemma \ref{lem: polydivision} on polynomial division, this implies that $\Psi^P \tPsi^P N_{ij}^P \in \PW(\fa^*)$.
	
	Now, $M_{ij} = \Psi^P \tPsi^P N_{ij}^P$ is the solution to $\Psi^P \tPsi^P u^P = \Phi^P M^P \tPhi^P$, so we have shown $M_{ij} \in \PW(\fa^*)$.
\end{proof}

For completeness, and to benefit the reader, we will summarize the proof of \cite[Lemma 7]{delormePW}. The claim is that $\Psi^P \tPsi^P N^P$ is holomorphic. It suffices to show that $\Psi^P \tPsi^P N^P$ is holomorphic on $\overline{\fa^*_{P,+}}$ for each $P$ (this is because each $N^P$ and $N^Q$ are intertwined by $\cA(Q,P,\lambda)$, which is nonzero and independent of $\lambda$ when acting on minimal $K$-types).

By writing $\cA(\overline{P},P,\lambda)$ as a product of $\cA(Q,R,\lambda)$ for adjacent $Q,R$, Property 3 of Definition \ref{defn: Esigma} implies that $u^P(\lambda) \cA(\overline{P}, P,\lambda)^{-1}$ extends to a holomorphic function on $\overline{\fa^*_{P,+}}$. Now, with notation as in the above proof,
\[u^P(\lambda)\cA(\overline{P}, P,\lambda)^{-1} = \Phi^P N^P \tPhi^P\cA(\overline{P}, P,\lambda)^{-1}. \]
The rational function $\tPhi^P\cA(\overline{P}, P,\lambda)^{-1}$ turns out to be a polynomial in $\lambda$ (we can commute the action of $\cA(\overline{P}, P, \lambda)^{-1}$ over to $\widetilde{V}$), and
\[\det \left(\tPhi^P\cA(\overline{P}, P,\lambda)^{-1}\right) = \tPsi^P(\lambda) \overline{b}^P(-\lambda).\]
Therefore, by Cramer's rule, and since $\bbar^P(-\lambda)$ is nonzero on $\overline{\fa^*_{P,+}}$,
\[N^P_{ij} = \frac{q_{ij}}{\Psi^P \tPsi^P},\]
where $q_{ij}$ is holomorphic on $\overline{\fa^*_{P,+}}$. This proves that $\Psi^P \tPsi^P N^P$ is holomorphic on $\overline{\fa^*_{P,+}}$.
\begin{proof}[Proof of the Factoring Theorem]
	Fix an element $u \in \PW_{\text{divis}}(\fa^*, \End(p_F I_\sigma))$. In the following, we will only consider $P=P_n$ and omit the corresponding superscripts (for example, we write $\Phi = \Phi^{P_n}$). Using Lemma \ref{lem: specializeddelorme}, choose several left and right factoring data $(\mu_i^{(m)}, v_i^{(m)}, \phi_i^{(m)})$ and $(\tmu_j^{(r)}, \tv_j^{(r)}, \tphi_j^{(r)})$ such that
	\[\sum_m\Psi_m \equiv 1, \quad \sum_r \tPsi_r \equiv 1.\]
	Let $M^{(m,r)}$ be the corresponding matrices as in Lemma \ref{lem: main factoring lemma}, and let $E_{ij}^{(m,r)} \in \Hom(I_\sigma(\tmu_j^{(r)}),I_\sigma(\mu_i^{(m)}))$ be corresponding ``elementary matrices'' between left and right factoring data. Then
	\begin{align*}
		u &= \sum \Psi_m u \tPsi_r = \sum \Phi_m M^{(m,r)} \tPhi_r\\
		&= \sum \pi_\sigma(\phi_i^{(m)}) M_{ij}^{(m,r)}E_{ij}^{(m,r)} \pi_\sigma(\tphi_j^{(r)}).
	\end{align*}
	
	We have shown that
	\[u \in  \pi_\sigma(R(\fg,F)p_{A(\sigma)}) \cdot \PW(\fa^*, \End(p_{\minK{\sigma}} I_\sigma))\cdot \pi_\sigma(p_{A(\sigma)}R(\fg,F)).\]
	Now, $u$ commutes with the action of $W_\sigma$ given by $\cA(P_n, w, \lambda)$ (this is Property 2 of Definition \ref{defn: Esigma}), and so if we average by this action we obtain
	\begin{align*}
		u &\in  \left[\pi_\sigma(R(\fg,F)p_{A(\sigma)}) \cdot \PW(\fa^*, \End(p_{\minK{\sigma}} I_\sigma))\cdot \pi_\sigma(p_{A(\sigma)}R(\fg,F))\right]^{W_\sigma}\\
		&= \pi_\sigma(R(\fg,F)p_{A(\sigma)}) \cdot \PW(\fa^*, \End(p_{\minK{\sigma}} I_\sigma))^{W_\sigma}\cdot \pi_\sigma(p_{A(\sigma)}R(\fg,F))\\
		&= \pi_\sigma(R(\fg, F)\cS(G,A(\sigma))R(\fg,F)),
	\end{align*}
	where Theorem \ref{thm: PW on lowktypes} is used in the second equality, and for the first equality we note that elements of $\pi_\sigma(R(\fg, F))$ commute with the action of $W_\sigma$.
	\end{proof}

\newpage

\bibliographystyle{alpha}
\bibliography{references}

\begin{thebibliography}{GAJV19}

\bibitem[Afg19]{afgoustidisCK}
Alexandre Afgoustidis.
\newblock On the analogy between real reductive groups and {C}artan motion
  groups: a proof of the {C}onnes-{K}asparov isomorphism.
\newblock {\em J. Funct. Anal.}, 277(7):2237--2258, 2019.

\bibitem[Art83]{arthur}
J.~Arthur.
\newblock A {P}aley-{W}iener theorem for real reductive groups.
\newblock {\em Acta Math.}, 150(1-2):1--89, 1983.

\bibitem[BCH94]{baumconneshigson}
Paul Baum, Alain Connes, and Nigel Higson.
\newblock Classifying space for proper actions and {$K$}-theory of group
  {$C^\ast$}-algebras.
\newblock In {\em {$C^\ast$}-algebras: 1943--1993 ({S}an {A}ntonio, {TX},
  1993)}, volume 167 of {\em Contemp. Math.}, pages 240--291. Amer. Math. Soc.,
  Providence, RI, 1994.

\bibitem[BH21]{2021braddhigson}
J.~Bradd and N.~Higson.
\newblock On {N}ovodvorskii’s theorem and the {O}ka principle.
\newblock {\em European Journal of Mathematics}, 03 2021.

\bibitem[BK14]{2014bernsteinkrotz}
J.~Bernstein and B.~Kr\"{o}tz.
\newblock Smooth {F}r\'{e}chet globalizations of {H}arish-{C}handra modules.
\newblock {\em Israel J. Math.}, 199(1):45--111, 2014.

\bibitem[Bos90]{1990bost}
J.-B. Bost.
\newblock Principe d'{O}ka, {$K$}-th\'{e}orie et syst\`emes dynamiques non
  commutatifs.
\newblock {\em Invent. Math.}, 101(2):261--333, 1990.

\bibitem[CCH16]{clarecrisphigson}
P.~Clare, T.~Crisp, and N.~Higson.
\newblock Parabolic induction and restriction via {$C^*$}-algebras and
  {H}ilbert {$C^*$}-modules.
\newblock {\em Compos. Math.}, 152(6):1286--1318, 2016.

\bibitem[CD84]{delormeclozel}
L.~Clozel and P.~Delorme.
\newblock Le th\'{e}or\`eme de {P}aley-{W}iener invariant pour les groupes de
  {L}ie r\'{e}ductifs.
\newblock {\em Invent. Math.}, 77(3):427--453, 1984.

\bibitem[CD90]{delormeclozel2}
Laurent Clozel and Patrick Delorme.
\newblock Le th\'{e}or\`eme de {P}aley-{W}iener invariant pour les groupes de
  {L}ie r\'{e}ductifs. {II}.
\newblock {\em Ann. Sci. \'{E}cole Norm. Sup. (4)}, 23(2):193--228, 1990.

\bibitem[Cow78]{cowlingKS}
Michael Cowling.
\newblock The {K}unze-{S}tein phenomenon.
\newblock {\em Ann. of Math. (2)}, 107(2):209--234, 1978.

\bibitem[Del84]{delormeHC}
P.~Delorme.
\newblock Homomorphismes de {H}arish-{C}handra li\'{e}s aux {$K$}-types
  minimaux des s\'{e}ries principales g\'{e}n\'{e}ralis\'{e}es des groupes de
  {L}ie r\'{e}ductifs connexes.
\newblock {\em Ann. Sci. \'{E}cole Norm. Sup. (4)}, 17(1):117--156, 1984.

\bibitem[Del05]{delormePW}
Patrick Delorme.
\newblock Sur le th\'{e}or\`eme de {P}aley-{W}iener d'{A}rthur.
\newblock {\em Ann. of Math. (2)}, 162(2):987--1029, 2005.

\bibitem[DFJ91]{delormeflensted}
P.~Delorme and M.~Flensted-Jensen.
\newblock Towards a {P}aley-{W}iener theorem for semisimple symmetric spaces.
\newblock {\em Acta Math.}, 167(1-2):127--151, 1991.

\bibitem[Dix77]{dixmiercstar}
Jacques Dixmier.
\newblock {\em {$C\sp*$}-algebras}.
\newblock North-Holland Mathematical Library, Vol. 15. North-Holland Publishing
  Co., Amsterdam-New York-Oxford, 1977.
\newblock Translated from the French by Francis Jellett.

\bibitem[DS04]{delormesouaifi}
P.~Delorme and S.~Souaifi.
\newblock Filtration de certains espaces de fonctions sur un espace
  sym\'{e}trique r\'{e}ductif.
\newblock {\em J. Funct. Anal.}, 217(2):314--346, 2004.

\bibitem[Ehr70]{ehrenpreisSCV}
Leon Ehrenpreis.
\newblock {\em Fourier analysis in several complex variables}, volume Vol. XVII
  of {\em Pure and Applied Mathematics}.
\newblock Wiley-Interscience [A division of John Wiley \& Sons, Inc.], New
  York-London-Sydney, 1970.

\bibitem[GAJV19]{valetteBCsurvey}
Maria~Paula Gomez~Aparicio, Pierre Julg, and Alain Valette.
\newblock The {B}aum-{C}onnes conjecture: an extended survey.
\newblock In {\em Advances in noncommutative geometry---on the occasion of
  {A}lain {C}onnes' 70th birthday}, pages 127--244. Springer, Cham, [2019]
  \copyright 2019.

\bibitem[Gra57a]{grauert1}
Hans Grauert.
\newblock Approximationss\"{a}tze f\"{u}r holomorphe {F}unktionen mit {W}erten
  in komplexen {R}\"{a}men.
\newblock {\em Math. Ann.}, 133:139--159, 1957.

\bibitem[Gra57b]{grauert2}
Hans Grauert.
\newblock Holomorphe {F}unktionen mit {W}erten in komplexen {L}ieschen
  {G}ruppen.
\newblock {\em Math. Ann.}, 133:450--472, 1957.

\bibitem[Gra58]{grauert3}
Hans Grauert.
\newblock Analytische {F}aserungen \"{u}ber holomorph-vollst\"{a}ndigen
  {R}\"{a}umen.
\newblock {\em Math. Ann.}, 135:263--273, 1958.

\bibitem[HC58]{HCinvariants}
Harish-Chandra.
\newblock Spherical functions on a semisimple {L}ie group. {I}.
\newblock {\em Amer. J. Math.}, 80:241--310, 1958.

\bibitem[Hig08]{higsonmackeyanalogy}
Nigel Higson.
\newblock The {M}ackey analogy and {$K$}-theory.
\newblock In {\em Group representations, ergodic theory, and mathematical
  physics: a tribute to {G}eorge {W}. {M}ackey}, volume 449 of {\em Contemp.
  Math.}, pages 149--172. Amer. Math. Soc., Providence, RI, 2008.

\bibitem[Kas88]{kasparov88}
G.~G. Kasparov.
\newblock Equivariant {$KK$}-theory and the {N}ovikov conjecture.
\newblock {\em Invent. Math.}, 91(1):147--201, 1988.

\bibitem[Kna82]{knappcommutativity}
A.~W. Knapp.
\newblock Commutativity of intertwining operators for semisimple groups.
\newblock {\em Compositio Math.}, 46(1):33--84, 1982.

\bibitem[Kna01]{knappoverview}
Anthony~W. Knapp.
\newblock {\em Representation theory of semisimple groups}.
\newblock Princeton Landmarks in Mathematics. Princeton University Press,
  Princeton, NJ, 2001.
\newblock An overview based on examples, Reprint of the 1986 original.

\bibitem[Kna02]{knappbeyondintro}
Anthony~W. Knapp.
\newblock {\em Lie groups beyond an introduction}, volume 140 of {\em Progress
  in Mathematics}.
\newblock Birkh\"{a}user Boston, Inc., Boston, MA, second edition, 2002.

\bibitem[KS80]{knappstein}
A.~W. Knapp and E.~M. Stein.
\newblock Intertwining operators for semisimple groups. {II}.
\newblock {\em Invent. Math.}, 60(1):9--84, 1980.

\bibitem[KV95]{knappvogan}
Anthony~W. Knapp and David~A. Vogan, Jr.
\newblock {\em Cohomological induction and unitary representations}, volume~45
  of {\em Princeton Mathematical Series}.
\newblock Princeton University Press, Princeton, NJ, 1995.

\bibitem[Nov67]{novodvorskii}
M.~E. Novodvorski\u{\i}.
\newblock Certain homotopic invariants of the space of maximal ideals.
\newblock {\em Mat. Zametki}, 1:487--494, 1967.

\bibitem[Par09]{paravicini}
Walther Paravicini.
\newblock Morita equivalences and {$KK$}-theory for {B}anach algebras.
\newblock {\em J. Inst. Math. Jussieu}, 8(3):565--593, 2009.

\bibitem[Phi91]{phillipsktheory}
N.~Phillips.
\newblock {$K$}-theory for {F}r\'{e}chet algebras.
\newblock {\em Internat. J. Math.}, 2(1):77--129, 1991.

\bibitem[Rai83]{rais}
Mustapha Rais.
\newblock Groupes lin\'{e}aires compacts et fonctions {$C^{\infty }$}
  covariantes.
\newblock {\em Bull. Sci. Math. (2)}, 107(1):93--111, 1983.

\bibitem[Var99]{varadarajanbook}
V.~S. Varadarajan.
\newblock {\em An introduction to harmonic analysis on semisimple {L}ie
  groups}, volume~16 of {\em Cambridge Studies in Advanced Mathematics}.
\newblock Cambridge University Press, Cambridge, 1999.
\newblock Corrected reprint of the 1989 original.

\bibitem[vdBS14]{vdbSouaifi}
Erik~P. van~den Ban and Sofiane Souaifi.
\newblock A comparison of {P}aley-{W}iener theorems for real reductive {L}ie
  groups.
\newblock {\em J. Reine Angew. Math.}, 695:99--149, 2014.

\bibitem[Vog79]{voganmainpaper}
David~A. Vogan, Jr.
\newblock The algebraic structure of the representation of semisimple {L}ie
  groups. {I}.
\newblock {\em Ann. of Math. (2)}, 109(1):1--60, 1979.

\bibitem[Vog81]{voganbook}
David~A. Vogan, Jr.
\newblock {\em Representations of real reductive {L}ie groups}, volume~15 of
  {\em Progress in Mathematics}.
\newblock Birkh\"{a}user, Boston, Mass., 1981.

\bibitem[Wal88]{wallachbook1}
Nolan~R. Wallach.
\newblock {\em Real reductive groups. {I}}, volume 132 of {\em Pure and Applied
  Mathematics}.
\newblock Academic Press, Inc., Boston, MA, 1988.

\end{thebibliography}

\end{document}